\documentclass[12pt]{amsart}
\usepackage{amsmath, amscd, color, url,graphicx}
\usepackage[all]{xy}

\usepackage{hyperref}
 
\theoremstyle{plain}
\newtheorem{thm}{Theorem}
\newtheorem{cor}{Corollary}

\newtheorem{lem}{Lemma}
\newtheorem{prop}{Proposition}

\theoremstyle{definition}
\newtheorem{defin}{Definition}
\theoremstyle{remark}
\newtheorem{rem}{Remark}
\theoremstyle{remark}
\newtheorem{notat}{Notation}
\theoremstyle{remark}
\newtheorem{example}{Example}
\theoremstyle{definition}
\newtheorem{assumption}{Assumption}

\newcommand{\N}{\mathbb{N}}
\newcommand{\Z}{\mathbb{Z}}

\newcommand{\R}{\mathbb{R}}
\newcommand{\C}{\mathbb{C}}
\newcommand{\SCS}{\mathcal{S}}
\newcommand{\HH}{\mathcal{H}}
\newcommand{\A}{\mathcal{A}}

\begin{document}
 
%Topmatter 

\title[Natural extensions and Gauss measures]{Natural extensions and Gauss measures for piecewise homographic continued fractions}

\author{ Pierre Arnoux}
\address{Institut de Math\'ematiques de Marseille (UMR7373),
          Site Sud, Campus de Luminy, Case 907
          13288 MARSEILLE Cedex 9
          France}
\email{pierre\@ pierrearnoux.fr}
\author{Thomas A. Schmidt}
\address{Oregon State University\\ Corvallis, OR 97331}
\email{toms@math.orst.edu}

\date{15 September  2017}
%\date{\today}
\thanks{The first author thanks the ANR projects SubTile and FAN for their support}
\thanks{Both of us take great pleasure in thanking Professor Mizutani for his friendly hospitality
during work on this project.}

%End topmatter
 
\begin{abstract}  We give a heuristic method to solve explicitly for an absolutely continuous invariant measure for 
a piecewise differentiable, expanding  map of a compact subset $I$ of Euclidean space $\mathbb R^d$.  The method consists of constructing a skew product family of maps on $ I \times \mathbb R^d$, which has an attractor. 
%(in the sense that associated to this family is a contractive map on the set of compacta fibering over $I$, the attractor is the fixed point).   
Lebesgue measure is invariant for the skew product family restricted to this attractor.   Under reasonable measure theoretic conditions,  integration over the fibers gives the desired measure on $I$.  Furthermore, the attractor system is then the natural extension of the original map with this measure.  We illustrate this method by relating it to various results in the literature.
\end{abstract}
 
\maketitle

\tableofcontents

%-------------------------------------------------------------------------------
%section 1 Intro
%-------------------------------------------------------------------------------
\section{Introduction} In a letter to Laplace in 1812, see the Appendix, Gauss  \cite{G}  gave the explicit formula for the invariant measure associated to regular continued fractions.  This is the measure that we refer to as the standard Gauss measure.    To this day,  it remains unclear how Gauss (using ``very simple reasoning") found this measure.  Keane \cite{Keane} suggests the possibility that Gauss used  a version of the natural extension of the interval map that was not explicitly introduced into the literature until 1977,  by Nakada-Ito-Tanaka \cite{NIT}.  This latter system was used for ergodic theoretic results by Bosma-Jager-Wiedijk (1983), \cite{BJW}.  

 Indirectly related to this  were the works showing that the Gauss system is a factor of the geodesic flow on the unit tangent bundle of the modular surface:  Artin \cite{Art} in 1924 seems to be the first to have shown this connection, followed among others by Adler-Flatto in 1984 (revisited in \cite{AF}), Series \cite{Se} in 1985, and Arnoux \cite{Ar} in 1994 (see also \cite{AN}, \cite {AL}).  A large number of authors have pursued these and related matters,   see for example    \cite{GroechenigHaas96}, \cite{MayerStroemberg08}, \cite{MayerMuehlenbruch10},  \cite{HilgertPohl09}.    For much of this history,  further motivation,  and also the work of S.~Katok and co-authors using various means to code geodesics on the modular surface,  see  \cite{KatokUgarcovici07}.

  We do not pretend to suggest that Gauss used the method we present here.  Still,  in the modern mathematical world this is an applicable, easily used,  method.      Indeed,  in Section~\ref{s:classExamples}, we give various examples of interval maps and higher dimensional maps for which our method determines an invariant density (thus a Gauss measure) and the corresponding  natural extension.     In particular, one easily recovers the natural extension as given by  Nakada-Ito-Tanaka and thus the invariant measure is rediscovered, see Subsection~\ref{ss:GaussFarey}.  

   In brief, we show that if an expanding piecewise M\"obius interval map is such that a naturally associated two dimensional system has a positive invariant Lebesgue measure,  then this system is the natural extension  of the interval map whose Gauss measure is the marginal measure of the two dimensional system.  In fact, the approach is not limited in dimension, see \cite{AN}, \cite{AL}.
    
   It seems to be part of the folklore that  natural extensions for certain interval maps arise as attractors.  In particular,  Katok and Ugarcovici \cite{KatokUgarcovici10}  explicitly prove that the planar models that they determine for the natural extensions  of the family of maps that they study are indeed attactors.     
 Our approach is to view the M\"obius functions as defining what can reasonably called a parametrized Iterated Function System (pIFS) on the interval cross the real line, and, generalizing Hutchinson's Theorem \cite{H}, show that this pIFS has a fixed compact set projecting onto the line.   This compact set gives the aforementioned two dimensional system,  and thus the Gauss measure for the interval map.      
 
     The pIFS arises in the following manner.   Suppose that our expanding continued fraction map $T$ is locally given  by a M\"obius transformation of matrix 
$M = \begin{pmatrix} 
a&b\\c&d\end{pmatrix}$ with determinant one, and consider the transformation  $\widetilde T$  locally defined by 
\begin{equation}\label{eq:2Daction}
{\mathcal  T}_M: (x,y) \to (\,M \cdot x,(cx+d)^2 y- c(cx+d)\,)\,.
\end{equation}  
This has  Jacobian matrix  determinant whose value is one and thus  each branch of 
 ${\mathcal  T}_M$ preserves  Lebesque measure.    Under a boundedness condition on the set of $c(cx+d)$, we prove  that there is a unique compact set $K$,  arising as an appropriate fixed point, 
 such that   $K$ equals the closure  of $\widetilde T(K)$.  If furthermore, $K$ and  $\widetilde T(K)$ have the same strictly positive measure, then an invariant density for $T$ is given by taking the 
 measure of the fiber of $K$ over any $x$ in the domain of $T$,  and  $\widetilde T(K)$ on $K$ defines a natural extension for $T$ and this invariant measure. 
 
As we have already shown in \cite{AS2}, in practice one can plot the orbit of almost any point under $\widetilde T(K)$ and use the resulting figure to compute equations for the boundary of $K$ and the invariant density.    Again, see Section~\ref{s:classExamples} for examples.    

\subsection{Main results}
We state and prove our fixed point theorem in a more general setting.   
   
\begin{thm} Suppose that $I$ is  a compact Hausdorff metric space, and $Y$ is a complete metric space.   
Let $\HH$ be the set of non-empty compact subsets of $Y$ with the usual Hausdorff metric, $d_H$.   Let $\SCS$ be the space of upper semi-continuous maps from $I$ to $\HH$,  $d_{\SCS}(K,K'):=\sup\{d_H\left(K(x), K'(x)\,\right)| x\in I\}$. 

Given
\begin{enumerate}
\item[(i.)]  a countable cover $\{I_{\alpha}\}_{\alpha\in\A} $ of $I$ by compact sets, 
\item[(ii.)]  a family $T$ of continuous maps $T_{\alpha} : I_{\alpha}\to I$ such that the family $\{T_{\alpha}\left(I_{\alpha}\right)\}_{\alpha\in\A} $ is a cover of $I$,  and 
\item[(iii.)] a uniformly contracting family of continuous maps $S_{\alpha}: I_{\alpha}\times Y \to Y$, that is totally bounded, 
\end{enumerate} 

\medskip 
define  
\begin{enumerate}
\item[(a.)]  for each $\alpha$,    
$\widetilde T_{\alpha} : I_{\alpha}\times Y\to I\times Y$ by $\widetilde T_{\alpha}(x,y)=\left(T_{\alpha}(x), S_{\alpha}(x,y)\right)$,  and 
\medskip

\item[(b.)]   $\Theta : \SCS\to \SCS; \,$ by $\Theta(K)=\overline{\cup_{\alpha \in \A} \; \widetilde T_\alpha(K_\alpha)}$.  
\end{enumerate} 

\bigskip 
Then the map $\Theta$ has a unique fixed point. 
\end{thm}

  A restricted version of this theorem has already been applied in a paper by Arnoux, Mizutani, Sellami \cite{AMS}.   

\bigskip 
    
The above leads to the calculation of explicit models of natural extensions in the following manner.
    
\begin{thm} Let $I$ a compact subset of $\R^d$, which is the closure of its interior, and let $Y=\R^d$.    Given (i-iii) as above, with 
\begin{enumerate}
\item[(i$'$.)] each $I_{\alpha}$ the closure of its interior, and   all non-trivial intersections $I_{\alpha} \cap I_{\beta}$ of Lebesgue measure zero, 

\item[(ii$'$.)]  each $T_{\alpha} : I_{\alpha}\to T_{\alpha}(I_{\alpha})\subset I$ a diffeomorphism, and the collection of Jacobian matrices  $T'_{\alpha}(x)$  uniformly expanding,

\item[(iii$'$.)]   each $S_{\alpha}$ given by $S_{\alpha}(x,y) = {{T'_{\alpha}}^{*}}^{-1}(x).y +f_{\alpha}(x)$, for some 
family $\{f_{\alpha}\}$ of piecewise continuous and uniformly bounded maps $f_{\alpha}: I_{\alpha}\to \R^d$.  Here 
 ${{T'_{\alpha}}^{*}}^{-1}$ denotes the inverse of the transpose of $T'_{\alpha}$. 
\end{enumerate} 

    Let $K$ be the unique fixed point of $\Theta$.   Let $\mu$   be the marginal measure for  $m$ on $K$ by way of $\pi:K\to I$.  Denote by  $ \mathcal B,\widetilde{\mathcal B}$   the  Borel $\sigma$-algebras of $I, K$, respectively.  
    
\bigskip     
    If the Lebesgue measure $m(K)$ of $K$ is not zero and $K \setminus \widetilde T(K)$ has measure zero, then $(K,\widetilde{\mathcal B}, m, \widetilde T)$ is a model for the natural extension of $(I, \mathcal B, \mu, T)$.
\end{thm}

%-------------------------------------------------------------------------------
%section 2 Technnical definitions:  Hausdorff metric on compact subsets of a complete metric space $E$ and function from a compact set to the compact nonempty subsets of $E$, and upper semi-continuity; the associate uniform distance
% Completeness of that space. 
%-------------------------------------------------------------------------------
\section{Basic definitions} 

\subsection{Hausdorff distance on compact sets}  

Let $(Y,d) = (Y, d_Y)$ be a complete metric space. If $A$ is a nonempty subset of $Y$, and $y\in Y$, we define the distance of $y$ to $A$ as $d(y,A):=\inf \{d(y,z)\,|\, z\in A\}$. For two subsets $A,B$ of $Y$, we define $\delta(A,B):=\sup_{y\in A} d(y,B)$. This is obviously not symmetric, since $\delta(A,B)=0$ if $A\subset B$, but symmetrization allows us to define the Hausdorff distance.

\begin{defin}The Hausdorff pseudo-distance $d_H$ between two nonempty subsets $A,B$ of $Y$ is defined by $d_H(A,B):=\max (\, \delta(A,B), \delta(B,A)\,)$. 
\end{defin}

Alternatively,  we introduce the following notation,  which will be useful later. 
\begin{notat}\label{n:aEpsilon}  For any nonempty set $A\subset Y$, define  
\[A_{\epsilon}=\{y\in Y\,|\, d(y,A)\le \epsilon\}.\]
\end{notat}

Thus,  $d_H(A,B)$ is the greatest lower bound of the $\epsilon$ such that $A\subset B_{\epsilon}$ and $B\subset A_{\epsilon}$.

\begin{lem} \label{lem:trans} Given $\epsilon, \delta >0 $ and  nonempty sets  $A,B,C$ such that $A\subset B_{\epsilon}$ and $B\subset C_{\delta}$, one has  $A\subset C_{\epsilon+\delta}$.
\end{lem}

\begin{proof} Let $y\in A$. Then we can find $z\in B$ such that $d(y,z)\le \epsilon$, and $w\in C $ such that $d(z,w)\le \delta$, hence $d(y,w)\le \epsilon+\delta$. This proves that $A\subset C_{\epsilon+\delta}$.
\end{proof}

 This implies that $d_H$ satisfies the metric inequality $d_H(A,C)\le d_H(A,B)+d_H(B,C)$; but the  pseudo-distance $d_H(A,B)$  is not a  true distance function, since it takes the value zero when evaluated at a set and its closure.  However, it becomes a distance when restricted to nonempty compact subsets of  $Y$.

\begin{notat}
We denote by $\HH$ the set of nonempty compact subsets of $Y$.
\end{notat}

The following proposition is classical, see \cite{Barn}, Theorem~ 7.1.

\begin{prop}
The set $\HH$ endowed with the Hausdorff pseudo-distance   $d_H$ is a complete metric space.
\end{prop}

We will make use of the following lemma, whose elementary proof we leave to the reader.

\begin{lem} \label{lem:union}Let $K$ be a compact set, and let $(A_{\alpha})_{{\alpha}\in \A}$  and $(B_{\alpha})_{{\alpha}\in \A}$  be two sequences of nonempty compact subsets of $K$ indexed by the same countable (possibly finite) set $\A$, such that $d_H(A_{\alpha},B_{\alpha})\le \epsilon$ for all $ {\alpha}\in   \A$. Let $A$ (resp. $B$) be the closure of the union of the $A_{\alpha}$ (resp. $B_{\alpha}$).  Then  $d_H(A,B)\le\epsilon$.
\end{lem}

\subsection{Upper semi-continuous functions} 
 
Let $I$ be a compact  Hausdorff space, and let $\pi:I\times Y\to I$ be the projection on the first coordinate. We now consider a compact subset  $K\subset I\times Y$ such that $\pi(K)=I$.   

One can view $K$  as a fiber bundle over $I$, whose fiber over $x$ we denote by $K(x)$.   Since each fiber is compact, we mainly view 
$K$ as a map  $K : I\to \HH, \quad x\mapsto K(x) :=  \{y\in  Y |(x,y)\in K\}$. This map has a remarkable property: it is upper semi-continuous. Let us recall the appropriate definition. 

\begin{defin} Let $I$ be a topological space, and let $(\HH, \le)$ be an ordered space in which any subset which has an upper (resp. lower) bound admits a least upper (resp. greatest lower) bound. Let $f: I\to \HH$ be a map, and suppose that $f$ is bounded on a neighborhood $V\ni x_0$. The limit superior of $f$ when $x$ tends to $x_0$ is the greatest lower bound over all neighborhoods $U\subset V$ of $x_0$ of the least upper bound for all $x\in U$ of $f(x)$.  In standard notation, we have  $\limsup_{x\to x_0}f(x)= \land_{U\subset V; x_0\in U} \lor_{x\in U}f(x)$.
\end{defin}

Note that   $f(x_0)\le \limsup_{x\to x_0}f(x)$. We can now define upper semi-continuous functions. 

\begin{defin} Let $I$ be a topological space, and let $\HH$ be an ordered space in which any subset which has an upper (resp. lower) bound admits a least upper (resp. greatest lower) bound.  We say that a map $f: I\to \HH$ is upper semi-continuous if for every $x_0\in I$ we have   $\limsup_{x\to x_0}f(x)=f(x_0)$.
\end{defin}

We return to the setting that $\HH$ is the space of non-empty compact subsets of a complete metric space $Y$.   
The inclusion relation on sets gives an order on $\HH$.  The greatest lower bound of a subset of $\HH$ is the intersection of the corresponding subsets of $Y$ if it is not empty, and the least upper bound is given by the closure of the union, if it is compact.  Thus,  $\HH$ satisfies the condition of the definition.   
 
  We restrict  to the case that  $I$  is a metric space.  
Using the notation introduced in Notation~\ref{n:aEpsilon}, we can reformulate the definition in this setting.  
 
\begin{lem} A map $K: I\to \HH$ is upper semi-continuous if and only if for any $x\in I$ and any $\epsilon >0$ there exists $\delta>0$ such that if $d(x,x')<\delta$, then $K(x')\subset K(x)_{\epsilon}$.
\end{lem}

The proof of the lemma is left to the reader. We can now  characterize the functions corresponding to compact subsets of $I \times Y$   with its product topology  in terms of the map sending points of $I$ to the corresponding fiber as a subset of $Y$.   

\begin{prop} A subset $K\subset I\times Y$ with $\pi(K)=I$ is compact if and only if  it defines an upper semi-continuous function $K: I\to \HH$.
\end{prop}

\begin{proof}   Suppose that $K$ is compact, and that the corresponding function is not upper semi-continuous. Then we can find a point $x\in I$,  $\epsilon>0$, and  a sequence $(x_i)_{i\in \N}$ converging to $x$ such that, for all $i\in \N$, $K(x_i)\notin K(x)_{\epsilon}$. For each $x_i$, we can hence find $y_i$ such that $d_Y(y_i,K(x)\,)>\epsilon$. Since $(x_i,y_i)\in K$ for all $n$, by compactness we can extract a subsequence which converges to a limit of the form $(x,y)$. We must have $d(y,K(x))\ge \epsilon$, hence $(x,y)$ cannot be in $K$; but this is impossible, since $K$ is closed. 

Conversely,  certainly if the map to fibers defined by $K$ does not take all of its values in $\HH$, then $K$ is not compact.  Now suppose that it does define a function $K: I\to \HH$, and that this function is upper semi-continuous. Fix some number $\epsilon>0$; for any $x$, we can find an open ball $B(x,\delta)\subset I$ such that, for all $x'\in B(x,\delta)$, $K(x')\subset K(x)_{\epsilon}$. By  compactness, we can cover $I$ with a finite number of balls $B(x_n,\delta_n)$, with $1 \le n \le N$.   Now  define the compact set $L=\cup_{n=1}^{N}   K(x_n)_{\epsilon}$, and note that for any $x\in I$, we have $K(x)\subset L$. That is, all of the fibers of $K$ are  contained in a single compact set. 

Furthermore, suppose that $(x_i,y_i)$ is a sequence of points in $K$ which converges to some $(x,y) \in I\times Y$. By semi-continuity, $d(y_i,K(x))$ tends to zero, hence $d(y,K(x)\,)=0$.   Since $K(x)$ is compact, this implies that $y\in K(x)$, hence $(x,y)\in K$.  Thus, $K$ is closed. Since $K$ is closed and contained in the compact set $I\times L$,  it is compact.  
\end{proof}

\begin{defin} We denote by $\SCS$ the space of upper semi-continuous maps from $I$ to $\HH$, endowed with the uniform distance  $d_{\SCS}(K,K'):=\sup\{d_H\left(K(x), K'(x)\,\right)|\, x\in I\}$. This is well defined, because  every  $K(x)$ is contained in  the compact subset of $Y$ which is  the projection of $K$ on the second coordinate, hence    the map $x \mapsto d_H(K(x), K'(x))$ is uniformly bounded.
\end{defin}

\begin{example} The space $\SCS$ is also the space of compact subsets of $I\times Y$ which project onto  $I$, endowed with a particular distance.    To aid the reader's intuition,  in the case of $I = [0,1]$,  and $Y = \mathbb R$, consider the sets 
\[
\begin{aligned} 
A &= \left([0, 1/2] \times [0,1]\right) \cup \left([1/2, 1]\times [0, 1/2] \right),\\
B &= \left([0, 1/2] \times [0,1+\epsilon]\right) \cup \left([1/2, 1]\times [0, 1/2] \right),\\
 C &= \left([0, 1/2 + \epsilon] \times [0,1]\right) \cup \left([1/2+\epsilon, 1]\times [0, 1/2] \right),
\end{aligned}
\]
giving $d_H(A,B) = d_H(A,C) = d_{\SCS}(A,B) = \epsilon$, but $d_{\SCS}(A,C) = 1/2$.
\end{example}

We will need the following property.

\begin{prop} The space $\SCS$   endowed with the uniform  distance is complete.
\end{prop}

\begin{proof} Suppose that $K_n$ is a Cauchy sequence in $\SCS$. Then for any $x\in I$, the sequence $K_n(x)$ is a Cauchy sequence in $\HH$; by completeness of $\HH$, it converges to a compact nonempty set $K(x)$. It remains to prove that the function associated to $K$ is upper semi-continuous. 

Let $\epsilon>0$; by definition, we can find an integer $N$ such that, for all $m,n>N$ and all $x\in I$,  $d_H(K_m(x),K_n(x))<\epsilon/3$. This implies in the limit that, for all $n>N$, $d_H(K_n(x), K(x))   \le   \epsilon/3$. Now, choose such an $n$, and $x\in I$. We can find $\delta$ such that, if $d_I(x',x)<\delta$, we have $K_n(x')\subset K_n(x)_{\epsilon/3}$. But we have $K(x')\subset K_n(x')_{\epsilon/3}$ and $K_n(x)\subset K(x)_{\epsilon/3}$.  By Lemma \ref{lem:trans}, we have $K(x')\subset K(x)_{\epsilon}$ as soon as $d_I(x',x)<\delta$, hence the limit sequence is in $\SCS$.
\end{proof}

%-------------------------------------------------------------------------------
%section 3 The main theorem : completeness lemma, fixed point theorem
%-------------------------------------------------------------------------------
\section{A fixed point theorem}

\subsection{The setting}

We consider a finite or countable covering $\{I_{\alpha}\}_{\alpha\in\A} $ of $I$ by compact sets, and a family $T$ of continuous maps $T_{\alpha} : I_{\alpha}\to I$ such that $\{T_{\alpha}\left(I_{\alpha}\right)\}_{\alpha\in\A} $ is a covering of $I$. We also consider a family of continuous maps $S_{\alpha}: I_{\alpha}\times Y \to Y$.   For any $x\in I_{\alpha}$, we denote by $S_{\alpha,x}$ the map $S_{\alpha,x} : Y\to Y$, $y\mapsto S_{\alpha}(x,y)$.

We suppose that the maps $S_{\alpha,x}$ are uniformly contracting; that is,  there exists a positive constant $c<1$ such that, for all $\alpha\in\A$, for all $x\in I_{\alpha}$, and for all $y,z\in Y$, $d_Y(S_{\alpha,x}(y), S_{\alpha,x}(z))\le c\;  d_Y(y,z)$. 
We now define a family $\widetilde T$ of (skew-product) maps $\widetilde T_{\alpha} : I_{\alpha}\times Y\to I\times Y$ by $\widetilde T_{\alpha}(x,y)=\left(T_{\alpha}(x), S_{\alpha}(x,y)\right)$.

\begin{notat} For $\alpha \in \A$ and $K\in \SCS$  we denote by $K_\alpha=\{(x,y)\in K \,|\, x\in I_\alpha\}$ the subset of $K$ which projects to $I_\alpha$.   (Since $\alpha$ belongs to $\A$ and not $\mathbb R$, this should cause no confusion with $K_{\epsilon}$.)  
\end{notat}

\begin{assumption} If the covering is infinite, we add the condition that, for any $K\in \SCS$, $\overline{\cup_{\alpha\in \A} \; \widetilde T_{\alpha}(K_{\alpha})}$ is compact.  This is equivalent to assuming the {\em boundedness condition}: that, for some (or equivalently, any) $y\in Y$,  the set $\{S_{\alpha}(x,y)\,|\,\alpha\in \A, x\in I_{\alpha}\}$ is totally bounded; this condition is automatically satisfied if $\A$ is finite. 
\end{assumption}

\begin{defin} The set map associated with the skew-product $\widetilde T$ is the map $\Theta : \SCS\to \SCS; \, \Theta(K)=\overline{\cup \; \widetilde T_\alpha(K_\alpha)}$
\end{defin}

\begin{rem} The set $\Theta(K)$ is in $\SCS$, because it is compact by hypothesis and it projects onto $I$ since the $T_{\alpha}\left(I_{\alpha}\right)$ cover $I$. 
\end{rem}

\subsection{The theorem}

Our first result is the following.

\begin{thm} The map $\Theta$ has a unique fixed point.
\end{thm}

\begin{proof} Let $K$ be an element of $\SCS$, and $K'$ be its image by $\Theta$. By definition, we have 
\[K'(x')=\overline{\cup_{\alpha,x; T_{\alpha}(x)=x'}\; {S_{\alpha,x}(\,K(x)\,)}}\]

\smallskip

Suppose that $K_1, K_2$ are two elements of $\SCS$, with $d_{\SCS}(K_1, K_2)=\delta$. Let $K'_1=\Theta(K_1)$, $K'_2=\Theta(K_2)$, and let $x'\in I$.   Let $\alpha, x$ be such that $T_{\alpha}(x)=x'$;  we have $d_H(K_1(x), K_2(x))\le\delta$. Since $S_{\alpha,x}$ is uniformly contracting by %a factor of 
at least $c$, we have $d_H\left(S_{\alpha,x}(\,K_1(x)\,), S_{\alpha,x}(K_2(x))\right)\le c\delta$. Using Lemma \ref{lem:union}, we deduce that $d_H(K'_1(x'), K'_2(x'))\le c \delta$.  Since this is true for all $x'$, we obtain $d_{\SCS}(K'_1,K'_2)\le c\delta$.

We have proved that $\Theta$ is a strict contraction by at least $c$ on $\SCS$; since this space is complete, by the Banach fixed point theorem, $\Theta$ has a unique fixed point. 
\end{proof}

%-------------------------------------------------------------------------------
%section 4 Additional results in the case of positive measure or non empty interior
%-------------------------------------------------------------------------------
\section{Heuristic models for the Gauss measure and natural extensions of continued fraction maps}

We will now apply the result of the previous section, with $Y=\R^d$ and $I$ a compact subset of $\R^d$.
We consider a piecewise uniformly expansive surjective map $T : I\to I$, on a compact nonempty subset $I\subset \R^d$ which is the closure of its interior.  

\begin{notat}
More precisely, we consider  a finite or countable covering $\{I_{\alpha}\}_{\alpha\in\A} $ of $I$ by compact sets  each the closure of its interior, such that the intersection $I_{\alpha}\cap I_{\beta}$ of two distinct elements of the covering has Lebesgue measure zero. 
We then suppose that there is a constant $C>1$ and  a countable collection of diffeomorphisms $T_{\alpha} : I_{\alpha}\to T_{\alpha}(I_{\alpha})\subset I$ with Jacobian matrix $T'_{\alpha}(x)$ uniformly expanding by at least $C$: for any $x\in I_{\alpha}$, and for any tangent vector $v$, $\|T_{\alpha}'(x).v\|\ge C \|v\|$, where we denote by $\|v\|$ the euclidian norm. We also suppose that $\bigcup_{{\alpha}\in \A} T_{\alpha}(I_{\alpha})=I$. 
\end{notat}

\begin{rem}The map $T$ is defined by $T(x)=T_{\alpha}(x)$ if $x\in I_{\alpha}$. Hence the map might be multiply defined on the intersection of two elements of the covering; however, this intersection has measure zero, and is irrelevant from the measurable viewpoint. 
\end{rem}

%\begin{lem} If the map $T$ is transitive, the dynamical system $(I,T)$ has a unique ergodic invariant measure which is absolutely %continuous with respect to Lebesgue measure.
%\end{lem}

Our goal is to determine an absolutely continuous invariant measure (a Gauss measure) $\mu$ for $T$, 
and to build a geometric model for the natural extension of the dynamical system $(I,\mathcal B, \mu,T)$, where $\mathcal B$ is the Borel $\sigma$-algebra.  (Similarly we will use $\widetilde{\mathcal B}$ for the Borel $\sigma$-algebra in the geometric model).

\begin{defin}
The {\em standard form with parameter $f$} for the natural extension of $T$, where $f$ is a family of piecewise continuous and uniformly bounded maps $f_{\alpha}: I_{\alpha}\to \R^d$,  is the map 
\[ 
\begin{aligned} 
\widetilde T: I\times \R^d&\to I\times \R^d\\
                                (x,y)&\mapsto \left(T_{\alpha}(x), {{T'_{\alpha}}^{*}}^{-1}(x).y +f_{\alpha}(x)\right),
\end{aligned} 
\]
where ${{T'_{\alpha}}^{*}}^{-1}$ denotes the inverse of the transpose of $T'_{\alpha}$ (and as usual we identify $\mathbb R^d$ with its dual real vector space). 
\end{defin}

\begin{prop} There exists a unique compact set $K$ which projects onto $I$ and is invariant under $\widetilde T$.
\end{prop}

\begin{proof} We can apply the theorem of the previous section, with $S_{\alpha,x}(y)= {{T'_{\alpha}}^{*}}^{-1}(x).y +f_{\alpha}(x)$; the maps $S_{\alpha,x}$ are by definition contractions of ratio at most $\frac 1C$, and they satisfy the boundedness condition, hence there is a unique $K\in \SCS$ such that $\Theta(K)=K$.
\end{proof}

Note that, for any given $\alpha$, the map $\widetilde T_{\alpha}$ clearly has Jacobian determinant 1, since its Jacobian matrix is a $2d\times 2d $ matrix of the form $\begin{pmatrix}T'_{\alpha}(x)&*\\0&{{T'_{\alpha}}^{*}}^{-1}(x)\end{pmatrix}$. Hence the Lebesgue measure of $K_{\alpha}$ is equal to the measure of its image by $\widetilde T_{\alpha}$.

\begin{rem}  See pp. 648--650 of \cite{AN} for a discussion of an antecedent of this approach.
\end{rem}

\begin{defin}\label{d:Measure}
Consider the measurable function $\phi: I \to \mathbb R$  given by $\phi(x) = m(\, K(x)\,)$.    We define the measure $\mu$ on $I$  by  
 \[ \mu(S) := \int_S\, \phi(x)\, dx\,,\]
 for $S$ any Borel set.
\end{defin}

\begin{thm} \label{th:natext}Let $K$ be the unique fixed point of $\Theta$. If the Lebesgue measure $m(K)$ of $K$ is not zero and $K \setminus \widetilde T(K)$ has measure zero (which always holds when $\A$ is finite), then $(K,\widetilde{\mathcal B}, m, \widetilde T)$ is a model for the natural extension of $(I, \mathcal B, \mu, T)$. 
\end{thm}

\begin{proof} By construction, $\widetilde T_\alpha : K_\alpha\to \widetilde T_\alpha(K_\alpha)$ is a bijection which preserves Lebesgue measure, hence $m(K_\alpha)=m(\widetilde T_\alpha(K_\alpha))$. We have $\bigcup K_\alpha = K$ and up to measure zero 
$K= \widetilde T(K)=   \bigcup \widetilde T_\alpha(K_\alpha) $; 
hence we have 

\[\sum_{\alpha \in \A} m(\widetilde T_\alpha (K_\alpha)) =\sum_{\alpha\in \A} m(K_\alpha)=  m(K) = m(\widetilde T(K)) = m(\, \bigcup \widetilde T_\alpha(K_\alpha) \,)
\,,\]
and  the sets $\widetilde T_\alpha(K_\alpha)$ must be disjoint in measure. This implies that $\widetilde T$ is a bijection (up to a set of measure zero) which preserves Lebesgue measure.

Since $\pi(K)=I$, the system $(I,\mathcal B, \mu, T)$ is a factor of $(K,m, \widetilde{\mathcal B}, \widetilde T)$. It remains to prove that $(K,\widetilde{\mathcal B}, m, \widetilde T)$ is isomorphic to the natural extension, and not a strict extension of it. 

An element of the natural extension defines a sequence $(x_n)_{n\in \Z}$ such that $T(x_n)=x_{n+1}$ (this sequence is the projection of the well-defined bi-infinite orbit of the element in the natural extension). Suppose such  a sequence  is the image of two distinct elements of $K$. This means that we can find two 
elements $(x_0,y_0)$ and $(x_0, y'_0)$ in $K$, and two sequences $(y_n)_{n\in \Z}$ and  $(y'_n)_{n\in \Z}$ such that, for all $n\in \Z$, $\widetilde T^n(x_0,y_0)=(x_n, y_n)$ and $\widetilde T^n(x_0,y'_0)=(x_n, y'_n)$. But since the first coordinate is the same for both sequences, at each step the same branch of $T$ is used,  and hence for any $n\in \Z$, we have $d(y_{n}, y'_n)\ge C d(y_{n+1}, y'_{n+1})$. 

If $y_0$ is different from $y'_0$, this implies that   $\lim_{n\to -\infty}d(y_{n}, y'_n)=\infty$, which is impossible since $K$ is compact. Hence the projection to the natural extension is one-to-one. 
\end{proof}

\begin{cor}
If $K \setminus \widetilde T(K)$ has measure zero and $K$ has nonempty interior, then $(K,\widetilde{\mathcal B}, m, \widetilde T)$ is a model for the natural extension of $T$. 
\end{cor}

\begin{proof}
This is immediate, since any open subset of $\R^{2d}$  has nonzero Lebesgue measure.
\end{proof}

\begin{rem}  The knowledge of $K$ gives us the density of an invariant measure.  Indeed, for $x \in I$, and $\phi$ as in Definition~\ref{d:Measure}, we have 
$\phi(x) = \sum_{x', T(x') = x} |\, \text{Jac}_{x'} T|^{-1}\, \phi(x')$.  
That is,  the function $\phi$    
satisfies  the classical Ruelle equation for an invariant density.  
\end{rem}

\begin{rem} There are (obvious) examples where $K$ has empty  interior: If the parameter $f$ of $\widetilde T$ is identically 0, then it is immediate that $K=I\times \{0\}$, which has zero Lebesgue measure. Hence we need a good choice of the parameter in order to obtain a domain $K$ with nonzero measure. The next section will give  examples of such a parameter in the case of   M\"obius transformations.
\end{rem}

%-------------------------------------------------------------------------------
%section 5 piecewise homographies and some conjugate forms useful for geometry or arithmetics
%-------------------------------------------------------------------------------
\section{Piecewise homographies} 

\subsection{A natural choice of parameter for the standard form}    

We now restrict to a special case.   Suppose that the map $T$ is piecewise homographic and locally increasing on a compact interval $I\subset \R$. More precisely, suppose that we have a finite or countable partition (up to extremities) of $I$ in compact subintervals $I_{\alpha}$, and that $T$ is given on $I_{\alpha}$ by an element  $M_{\alpha}\in \text{PSL}(2,\R)$.

\begin{rem} To compute, we must represent $M_{\alpha}$ by a matrix $\begin{pmatrix} a&b\\c&d\end{pmatrix}\in \text{SL}(2,\R)$; this matrix is well-defined up to a sign, and we can suppose that $c\ge 0$; we will abuse notations by identifying $M_{\alpha}$ and this representative, which is uniquely defined if $c\ne 0$.
\end{rem}
 
For this matrix  $M_{\alpha}$, we have $T(x)=\frac{ax+b}{cx+d}$ and $T'(x)=\frac 1{(cx+d)^2}$. We suppose that $T$ is uniformly expansive, that is, there exists some constant $k<1$ such that $(cx+d)^2\le k$. To use the previous section, we must consider a map $\widetilde T(x,y)=\left (\frac{ax+b}{cx+d}, (cx+d)^2y+f(x)\right)$; the problem is to choose a good function $f$; as we have seen, the constant zero function will not work.

We will find an heuristic for  our function by supposing first that the elements  $M_{\alpha}$ generate a lattice $\Gamma \subset \text{PSL}(2,\R)$, that is, a discrete subgroup of $\text{PSL}(2,\R)$ such that $\Gamma\backslash \text{PSL}(2,\R)$ has finite volume. 

We recall here some results from \cite{AS}, Section 3. 
Let $A_{\gamma}\subset \text{PSL}(2,\R)$ be the set 
\[A_{\gamma}=\left\{\begin{pmatrix}\alpha&\beta\\\gamma&\delta\end{pmatrix}|\  \gamma> 0\right\}.\] 
  The set $A_{\gamma}$ has full Haar measure, as it is the complement of the  codimension 1 
 set  defined by $\gamma=0$.  On $A_{\gamma}$, we can take $\alpha, \gamma, \delta$ as coordinates, and the Haar measure   of $\text{PSL}(2,\R)$ is given on $A_{\gamma}$, up to a constant, by $\frac {d\alpha\, d\gamma\, d\delta}{|\gamma|}$. 

The diagonal (or {\em geodesic}) flow  is the action on  $\text{PSL}(2,\R)$  on the right by the one-parameter diagonal  group $g_t= \begin{pmatrix} e^t&0\\0&e^{-t}\end{pmatrix}$.  This preserves Haar measure, and commutes with the action of $\Gamma$ on the left, hence acts on the quotient $\Gamma\backslash \text{PSL}(2,\R)$. 

We will obtain $\widetilde T$ as a return map of the diagonal flow to a section. Any element in $A_{\gamma}$ can be written $$\begin{pmatrix}xe^t&(xy-1)e^{-t}\\e^t&ye^{-t}\end{pmatrix}$$

An easy computation shows that, in these coordinates, the Haar measure is given by $dx\, dy\, dt$. The matrices of the form $\begin{pmatrix} x& xy-1\\1&y \end{pmatrix}$ form a transversal of the diagonal flow on $A_{\gamma}\,$: any orbit of the diagonal flow on $A_{\gamma}$ meets this set in one point. We can now compute the first return map to the corresponding section for the diagonal flow on the quotient space. The product of this matrix by $M_{\alpha}$ is $\begin{pmatrix}ax+b& a(xy-1)+by\\cx+d&c(xy-1)+d\end{pmatrix}$; its orbit under the diagonal flow returns to the section at the point $\begin{pmatrix}\frac{ax+b}{cx+d}&*\\1&(cx+d)^2 y-c(cx+d)\end{pmatrix}$, and the transverse invariant measure is given by $dx\, dy$. 

We thus are lead to define the function $\widetilde T$ by $$\widetilde T(x,y):=\left(\frac{ax+b}{cx+d}, (cx+d)^2y-c(cx+d)\right).$$ This map is a local diffeomorphism with Jacobian 1.  If the function $c(c x +d)$ is uniformly bounded on $I$,  the results of  the previous section apply: There is a unique compact set $K$ such that $\overline{\widetilde T(K)}=K$.   If this compact set has nonzero measure and satisfies the condition  $K\setminus \widetilde T(K)$ has measure zero, then $\widetilde T: K\to K$ is a natural extension for $T$, which preserves Lebesgue measure on $K$. This allows one to compute an explicit formula for an invariant measure for  $T$  that is absolutely continuous with respect to Lebesgue measure.

We will  show in the next section that this heuristic works in a large number of cases.

\subsection{Some conjugate forms of the natural extension map}

The parametrization above has the nice feature that the invariant measure appears as Lebesgue measure, and the inconvenience that the geometric interpretation is unclear. 

We can find other sections of the diagonal flow by considering its geometric interpretation as the geodesic flow on the hyperbolic plane. Recall that one can identify $\text{PSL}(2,\R)$ with the unit tangent bundle of the hyperbolic plane; any matrix $\begin{pmatrix}\alpha&\beta\\\gamma&\delta\end{pmatrix}$ determines a unique geodesic, with origin $u=\frac{\beta}{\delta}$ and extremity $x=\frac{\alpha}{\gamma}$. The highest point on this geodesic corresponds to a matrix such that $\gamma=\delta$; the endpoints of the geodesic, plus the distance to the highest point, determine another coordinate system on $A_{\gamma}\,$: each element can be written in a unique way
$$\begin{pmatrix}
\frac{xe^s}{\sqrt{x-u}}  &   \frac{ue^{-s}}{\sqrt{x-u}}\\ 
\frac{e^s}{\sqrt{x-u}}    &    \frac{e^{-s}}{\sqrt{x-u}}
\end{pmatrix}.$$

The complex coordinate of the highest point of the corresponding geodesic is $\frac{x+u}2+i \frac{x-u}2$, and in these coordinates the Haar measure is given by $\frac{dx\, du\, ds}{(x-u)^2}$. Taking $s=0$, we obtain another section parametrized as 
$\begin{pmatrix}
\frac{x}{\sqrt{x-u}}  &   \frac{u}{\sqrt{x-u}}\\ 
\frac{1}{\sqrt{x-u}}    &    \frac{1}{\sqrt{x-u}}
\end{pmatrix}$ which can be identified with the hyperbolic plane; in these coordinates, the first return map to the section is given by $(x,u)\mapsto (\frac{ax+b}{cx+d},\frac{au+{b} }{cu+d})$, and the invariant measure turns out to be the natural hyperbolic measure, with density $\frac{dx\, du}{(x-u)^2}$. 

A variant quite common in number theory consists in replacing the coordinate $u$ by $v=-\frac 1u$. In that case, we consider the same section, but parametrized as 
 $$\begin{pmatrix}\frac x{\sqrt{x+\frac 1v}}&\frac {-1}{v\sqrt{x+\frac 1v}}   \\\frac 1{\sqrt{x+\frac 1v}}& \frac 1{\sqrt{x+\frac 1v}}\end{pmatrix}.$$ 

A small computation gives us another heuristic formula in this case: 

$$\widehat T(x,v):=\left(\frac{ax+b}{cx+d}, \frac {dv-c}{a-bv}\right).$$

One can easily check  that this formula can also be written $\widehat T(x,v):=\left(M.x, {M^{t}}^{-1}.v\right)$; $\widetilde T$ and  $\widehat T$ are conjugate by the map $(x,y)\mapsto (x, \frac{y}{1+xy})$, and the map $\widehat T$ leaves invariant the density $\frac {dx\, dv}{(1+xv)^2}$.

  Note that these two alternate forms of the natural extension map are locally products, as opposed to honest skew-product maps.  As well, $\widehat T$ sends rectangles to rectangles,  making it  particularly convenient.   As we will see below, the domain of $\widehat T$ has a local product structure, whereas that of $\widetilde T$ need not.  On the other hand,  the action of $\widehat T$ on its second coordinate is not linear, which raises difficulties in proving that it is   uniformly contracting on this coordinate. 

%-------------------------------------------------------------------------------
%section 6 Some examples: the basic counter exemple : beta expansions; generalised Rosen, a 1-parameter family of Gauss maps; Hurwitz continued fraction; chi-Han Sah map.
%-------------------------------------------------------------------------------
\section{Examples: classical and other}\label{s:classExamples}

In this section we give a number of examples; some of them already appeared in the literature,
in which case we only present the results and give references for the proofs.

\subsection{The  Gauss map and the Farey map}\label{ss:GaussFarey}

\begin{figure}
\begin{center}
\scalebox{0.4}{
\includegraphics{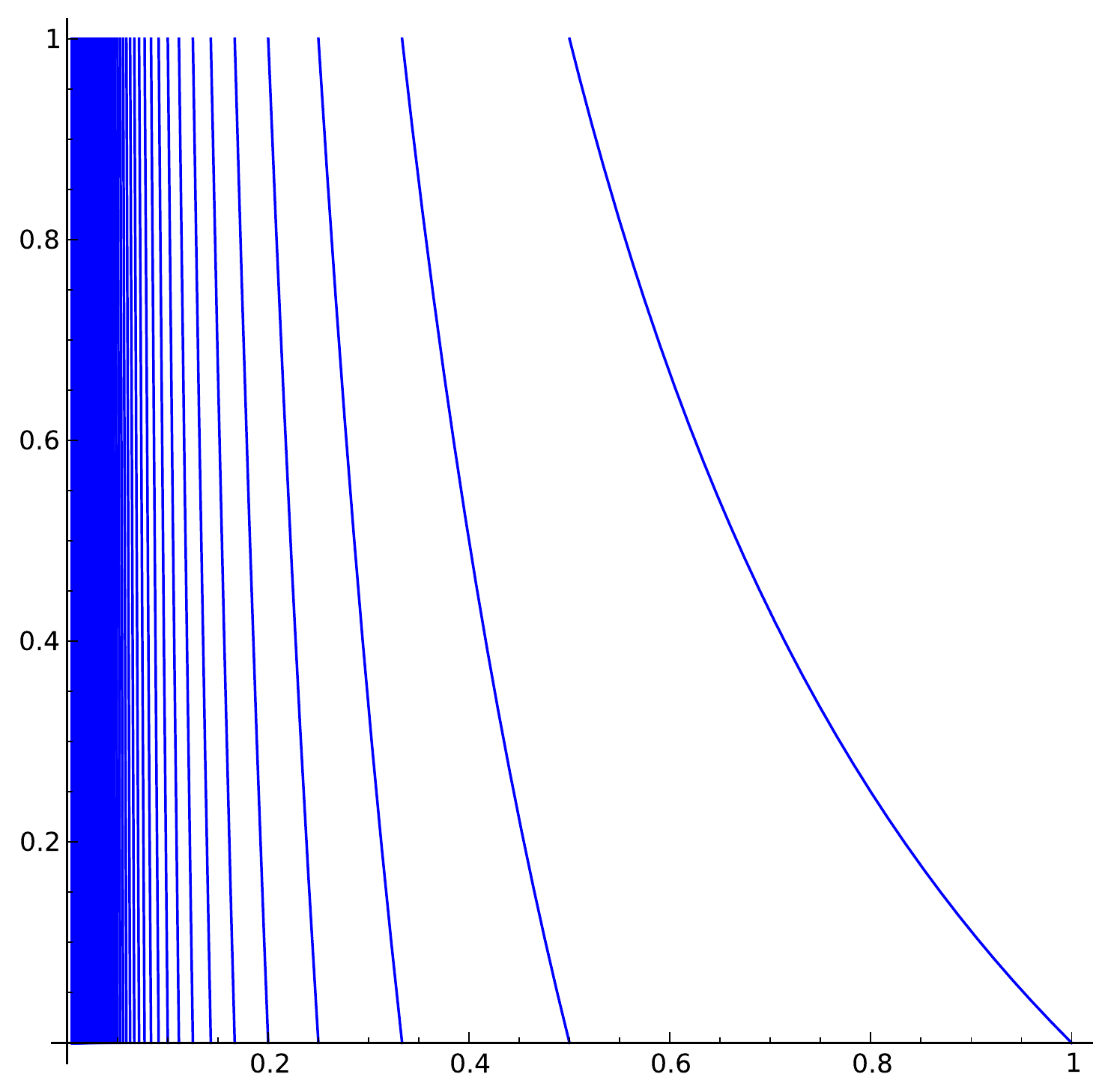}}
\caption{The graph of the Gauss map.}
\label{fig.Gaussmap}
\end{center}
\end{figure}

The classical Gauss map has been studied in many papers since the letter of Gauss, see for example \cite{Art, AF, Se, Ar}; it is related to the classical Euclid algorithm, replacing a pair $(a,b)$ with $a<b$ by $(b\mod a, a)$ and renormalizing at each step so that the largest number is 1. More precisely, if we denote the integer part of $x$ by  $[x]=\sup\{n\in \N |\, n\le x\}$, and the fractional part of $x$ by $\{x\}=x-[x]$, the Gauss map is defined as 

\[T: [0,1]\to [0,1]; \quad x\mapsto \left\{\frac 1x\right\}.\]
It is locally given by a homography $x\mapsto \frac 1x-n$, related to the matrix $\begin{pmatrix}-n&1\\1&0\end{pmatrix}$.  The map itself is not strictly contracting, since the derivative at $x=1$ is 1, but its square is strictly contracting; it satisfies the conditions of the theorem, hence there is a unique compact set $K$ invariant by the natural extension 

$$\widetilde T(x,y)= \left(\left\{\frac 1x\right\},x-x^2y\right).$$

This invariant set is easily found by a numerical experiment, plotting the initial part of the orbit of a generic point: it is the set $K=\{(x,y)\,|\,0\le x\le 1, 0\le y\le \frac 1{1+x}\}$, see Fig.~\ref{fig.Gaussmeasure1}. This set is decomposed in a countable  Markov partition $K_n= \{(x,y)\,|\,\frac 1{n+1}\le x\le \frac 1n, 0\le y\le \frac 1{1+x}\}$, see  Fig.~\ref{fig.Gaussmarkov}.

\begin{figure}
\begin{center}
\scalebox{0.5}{
\includegraphics{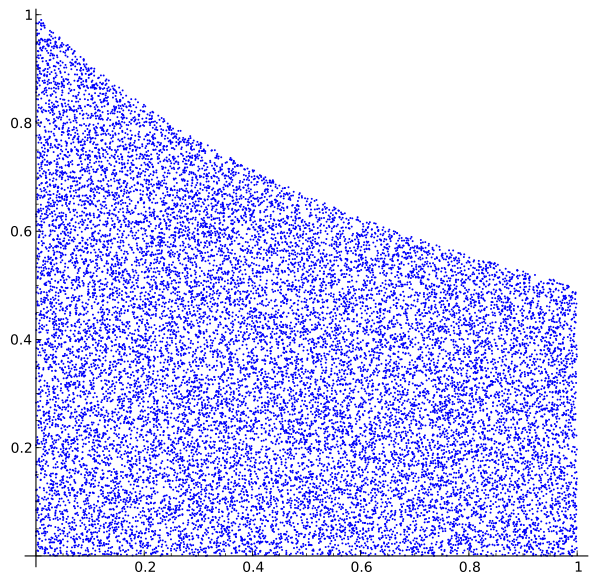}}
\caption{The domain of $\widetilde T$: $20, 000$ points of an orbit.}
\label{fig.Gaussmeasure1}
\end{center}
\end{figure}

\begin{figure}
\begin{center}
\scalebox{0.3}{
\includegraphics{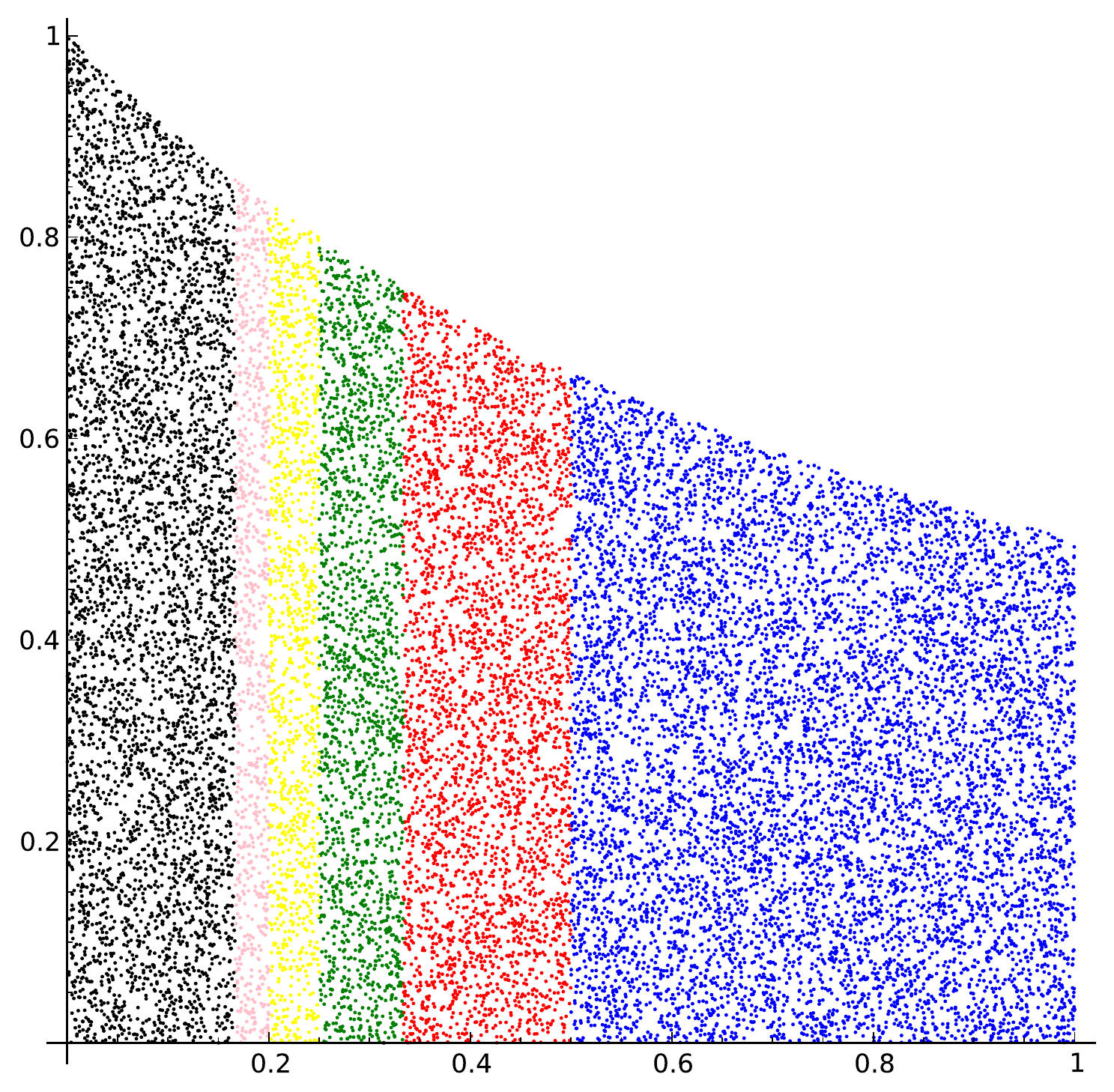}\includegraphics{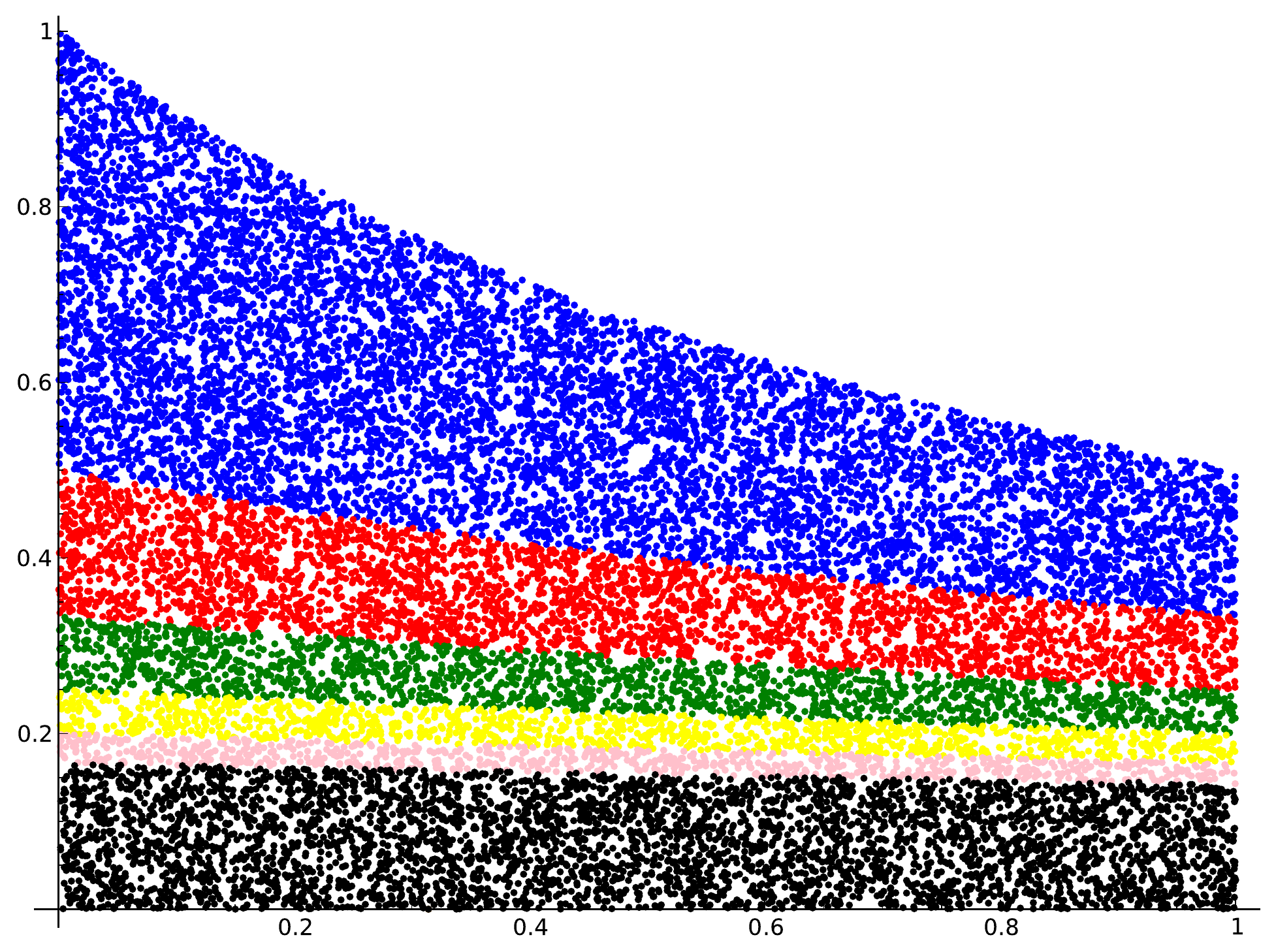}}
\caption{The Markov partition  of $\widetilde T$: five boxes and their image.}
\label{fig.Gaussmarkov}
\end{center}
\end{figure}

Alternatively, one can consider the other presentation of the natural extension; if we denote by $a(x):=[\frac 1x]$ the first partial quotient, it is defined by 

$$\widehat  T: [0,1]^2\to[0,1]^2\quad  (x,u)\mapsto  \left(\left\{\frac 1x\right\},\frac 1{a(x)+u}\right).$$

\begin{figure}
\begin{center}
\scalebox{0.5}{
\includegraphics{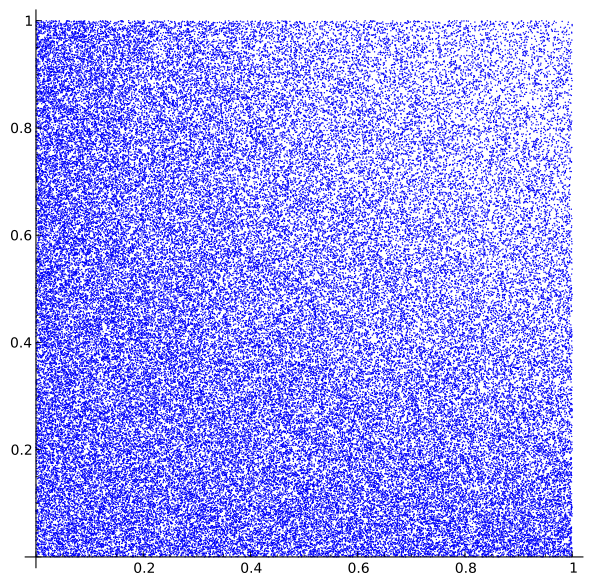}}
\caption{The domain of $\widehat T$: $20, 000$ points of an orbit. Note that the point distribution of the orbit is not uniform, compare with Fig.~\ref{fig.Gaussmeasure1}. }
\label{fig.Gaussmeasure2}
\end{center}
\end{figure}

This map preserves the invariant density $\frac {dx\, du}{(1+xu)^2}$, and we can see on Fig.~\ref {fig.Gaussmeasure2} that the point distribution of a generic orbit is not uniform.   The map is very natural from an arithmetic viewpoint: it has a countable Markov partition, $[\frac 1{n+1},\frac 1n]\times [0,1]$, the symbolic dynamic related to this partition is the full shift, and the bi-infinite symbolic dynamic is given by joining the continued fraction expansions of $x$ and $u$.  

Remark that this map is locally decreasing, hence it is not defined by a map of $\text{PSL}(2,\R)$ (but its square is). This is a problem which occurs frequently in the classical theory of continued fractions, and which explains why, in many theorems, one must consider parity of indices. There is a way to get rid of this by symmetrizing the Gauss map: define the {\em Truncation map } on $\R$ by $\text{trunc}(x)=\{x\}$ if $x>0$ and $\text{trunc}(x)=-\{-x\}$ if $x<0$, and the {\em symmetrized Gauss map}Ê$S$ on $[-1,1]$ by $S(x):=\text{trunc}(-1/x)$; this is a locally increasing function defined by elements of $SL(2,\R)$. Using the formula given above, one define a model for the natural extension as $\tilde S(x,y)=(S(x), x^2y-x)$; computation of a generic orbit immediately gives the domain for the natural extension of $S$, which factors  by symmetry on the natural extension of $T$; see  Fig.~\ref{fig.Gausssymmap}

\begin{figure}
\begin{center}
\scalebox{0.3}{
\includegraphics{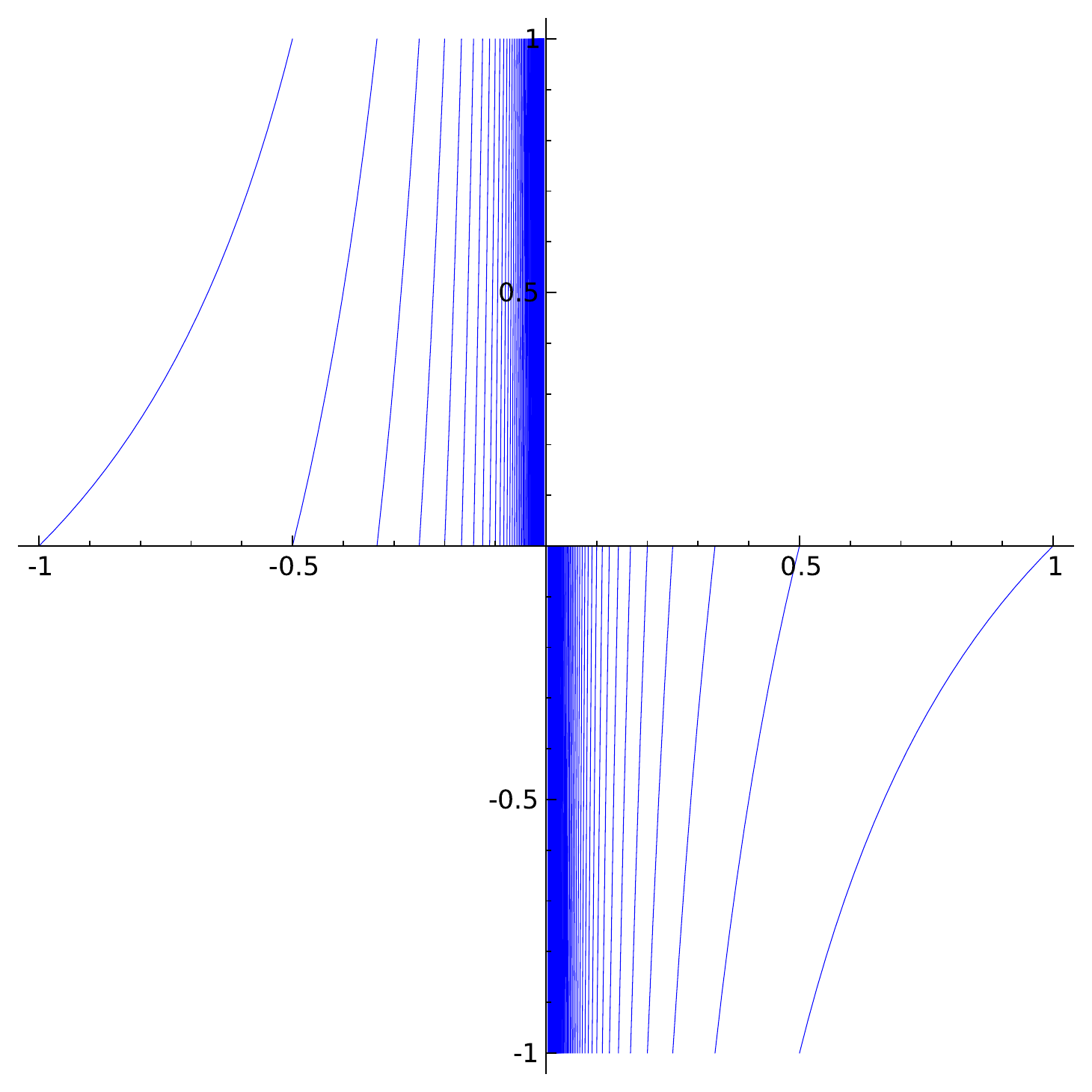} \phantom{a large space} \includegraphics{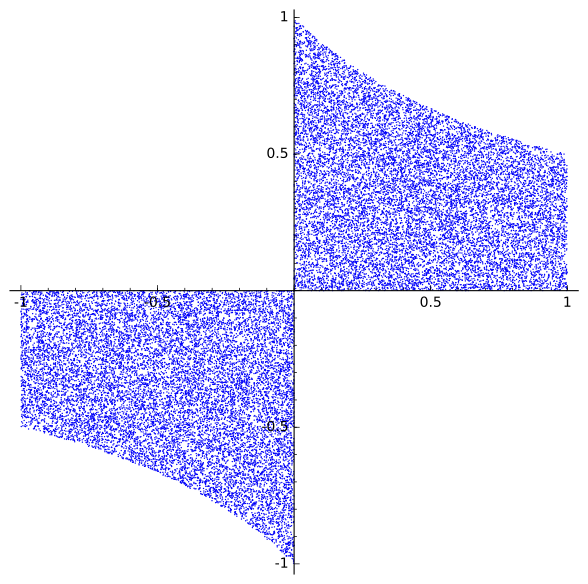}}
\caption{The symmetrized Gauss map, and $20, 000$ points of an orbit of its natural extension.}
\label{fig.Gausssymmap}
\end{center}
\end{figure}

The Gauss map admits an additive version, related to the additive euclidean algorithm, of which it is an acceleration, the {\em Farey map} $F$, defined on $[0,1]$ by $F(x)=\frac x{x-1}$ if $x<\frac 12$, and $F(x)=\frac 1x-1$ if $x>\frac 12$. Using the above formula, one can find a natural extension given by $\tilde F(x,y)=(\frac x{x-1}, (1-x)^2y+1-x)$ if $x<1/2$, and  $\tilde F(x,y)=(\frac 1x-1, x-x^2y)$ if $x>\frac 12$. This map is no longer expanding: it has an indifferent point at the origin, so our theorem does not apply in this case, and there is no invariant compact set. However, computer experiment immediately shows that the closed set $\{(x,y)\in \R^2|0\le x\le 1, 0\le y\le\frac 1x\}$ is invariant, and gives the infinite invariant measure $\frac {dx}x$, see Fig.~\ref{fig.Fareymap}. The indifferent fixed point precludes the possibility of an absolutely continuous invariant measure. 

\begin{figure}
\begin{center}
\scalebox{0.3}{
\includegraphics{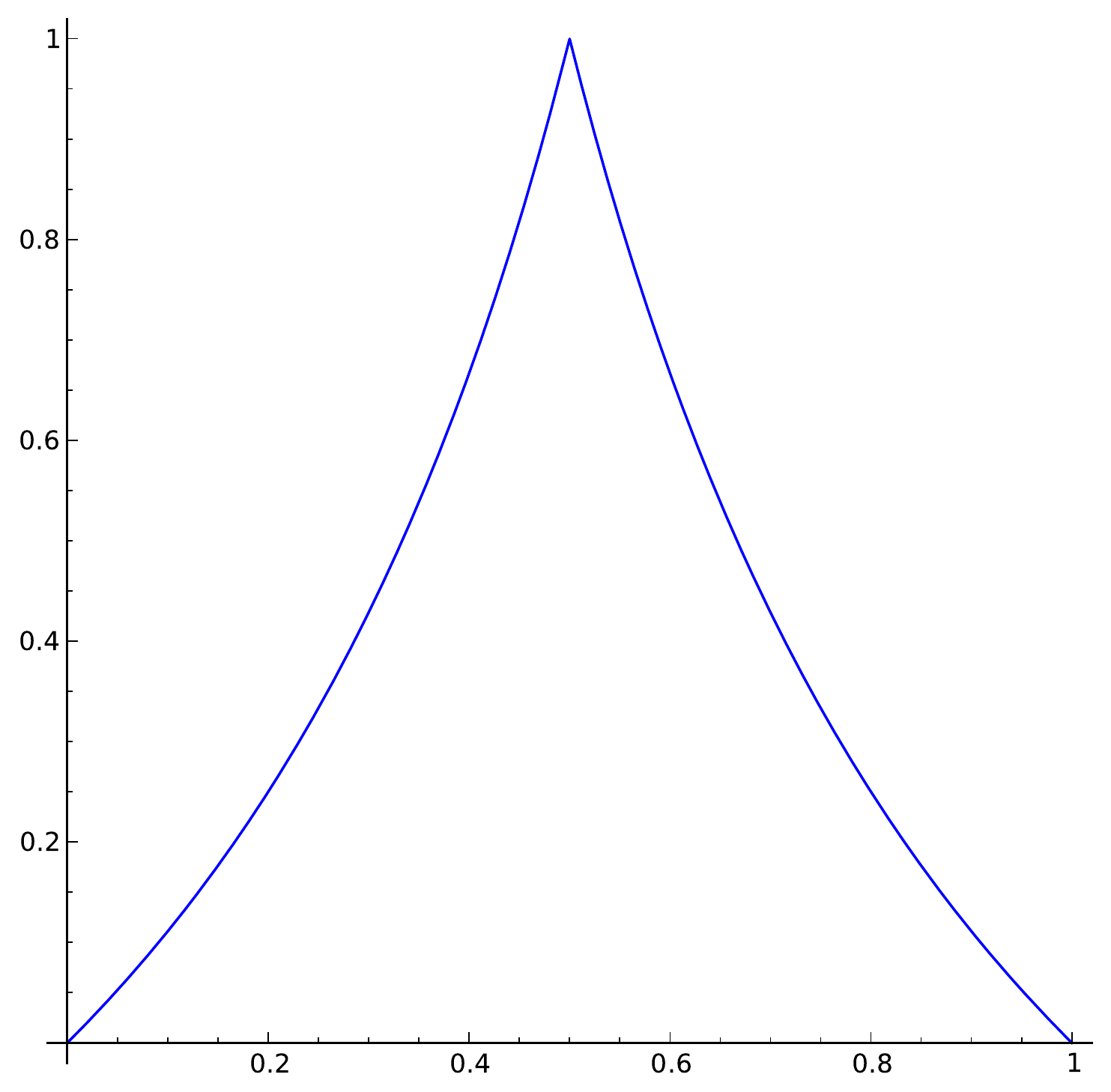} \phantom{a large space} \includegraphics{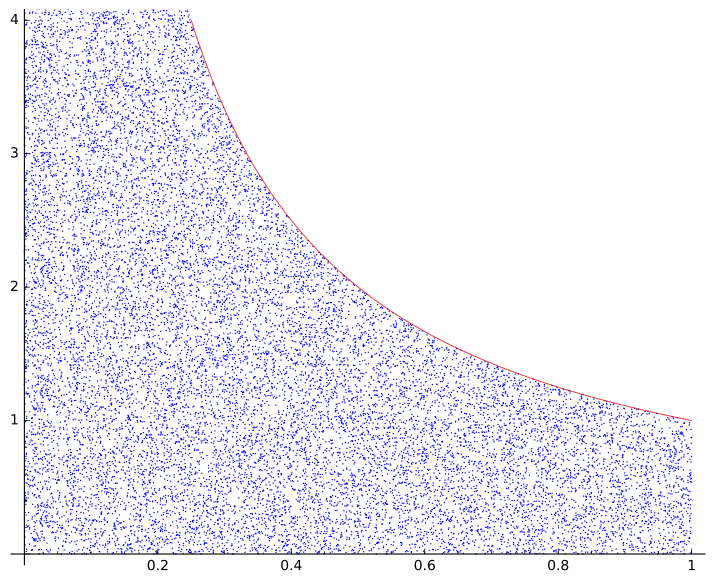}}
\caption{The Farey map, and $20, 000$ points of an orbit of its natural extension.}
\label{fig.Fareymap}
\end{center}
\end{figure}

The  Farey maps also admits an orientation preserving version $F^+$, defined by the same formula as $F$ for $x<\frac 12$, and by $F^+(x)=2-\frac 1x$ for $x>\frac 12$; its natural extension is given as before by $\widetilde F^+(x,y)=(\frac x{x-1},(1-x)^2y+1-x) $ for $x<\frac 12$, and by $\widetilde F^+(x,y)=(2-\frac 1x,x^2y-x) $ for $x>\frac12$. This map leaves invariant the domain $\{(x,y)\in \R^2|0\le x\le 1, \frac 1{x-1}\le y\le\frac 1x\}$, which gives for $F^+$ the invariant density $\frac 1{x(1-x)}$: this density is infinite in 0 and 1, because these are now two indifferent fixed points, see Fig.~\ref{fig.Farey2map}.

\begin{figure}
\begin{center}
\scalebox{0.3}{
\includegraphics{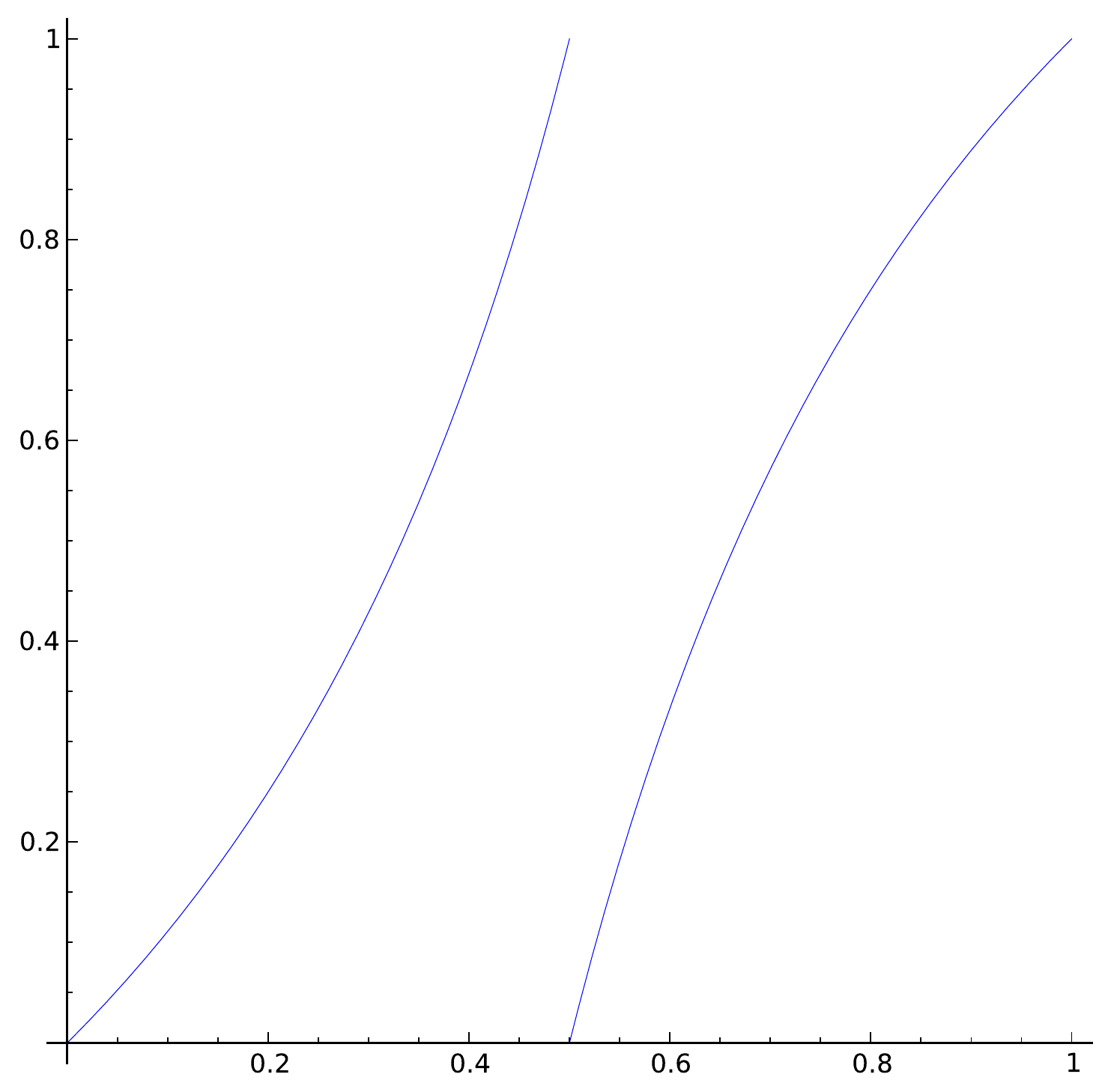}\phantom{a large space} \includegraphics{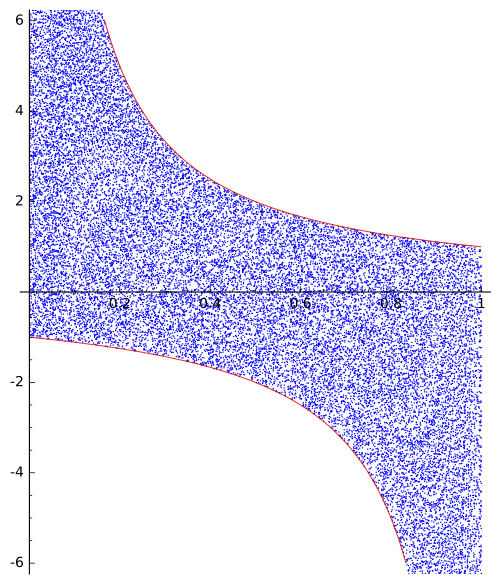}}
\caption{The orientation preserving Farey map, and $20, 000$ points of an orbit of its natural extension.}
\label{fig.Farey2map}
\end{center}
\end{figure}

\begin{rem} It is possible to conjugate the Farey map and the orientation preserving Farey map in such a way that the invariant density becomes constant; in this way, we obtain two curious maps; the first one is defined on $(0,+\infty)$ by 
$$x\mapsto |\log(e^x-1)|$$
and the second is defined on $\R$ by 
$$x\mapsto \text{sign}(x) \log(e^{|x|}-1)$$
It is readily verified, using the Ruelle equation, that these two maps preserve the (infinite) Lebesgue measure on their domain, and it is a consequence of the dynamical definition that this measure is ergodic; the maps are very close to the identity on most of their domain,  see Fig.~\ref{fig.Fareyconj}. One easily gives models for the natural extensions of these two maps, in forms of functions of two variables which leaves invariant $(0,+\infty)\times[-1,1]$ for the first map and $\R\times [-1,1]$ for the second.
\end{rem}

\begin{figure}
\begin{center}
\scalebox{0.3}{
\includegraphics{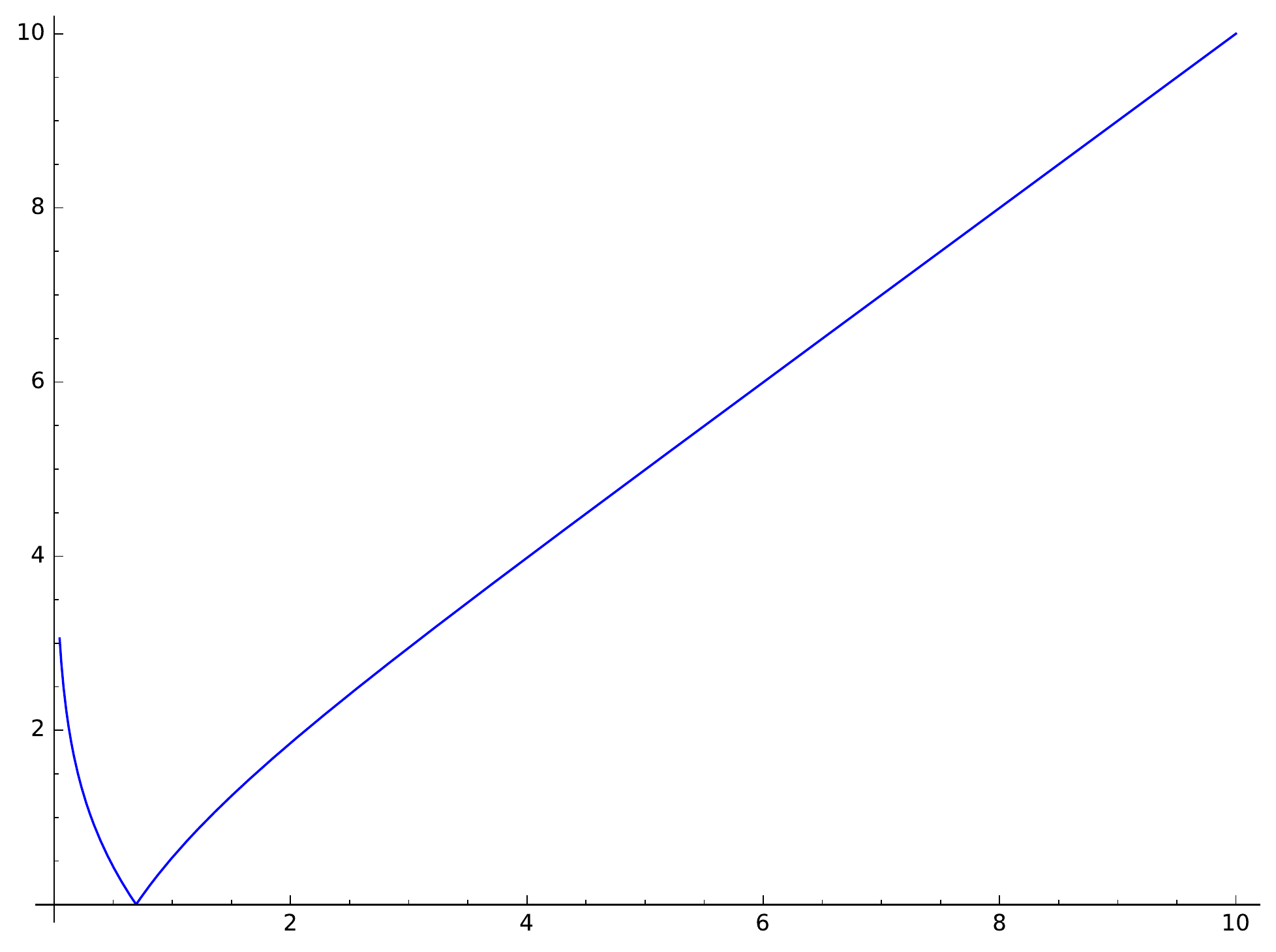}\phantom{a large space} \includegraphics{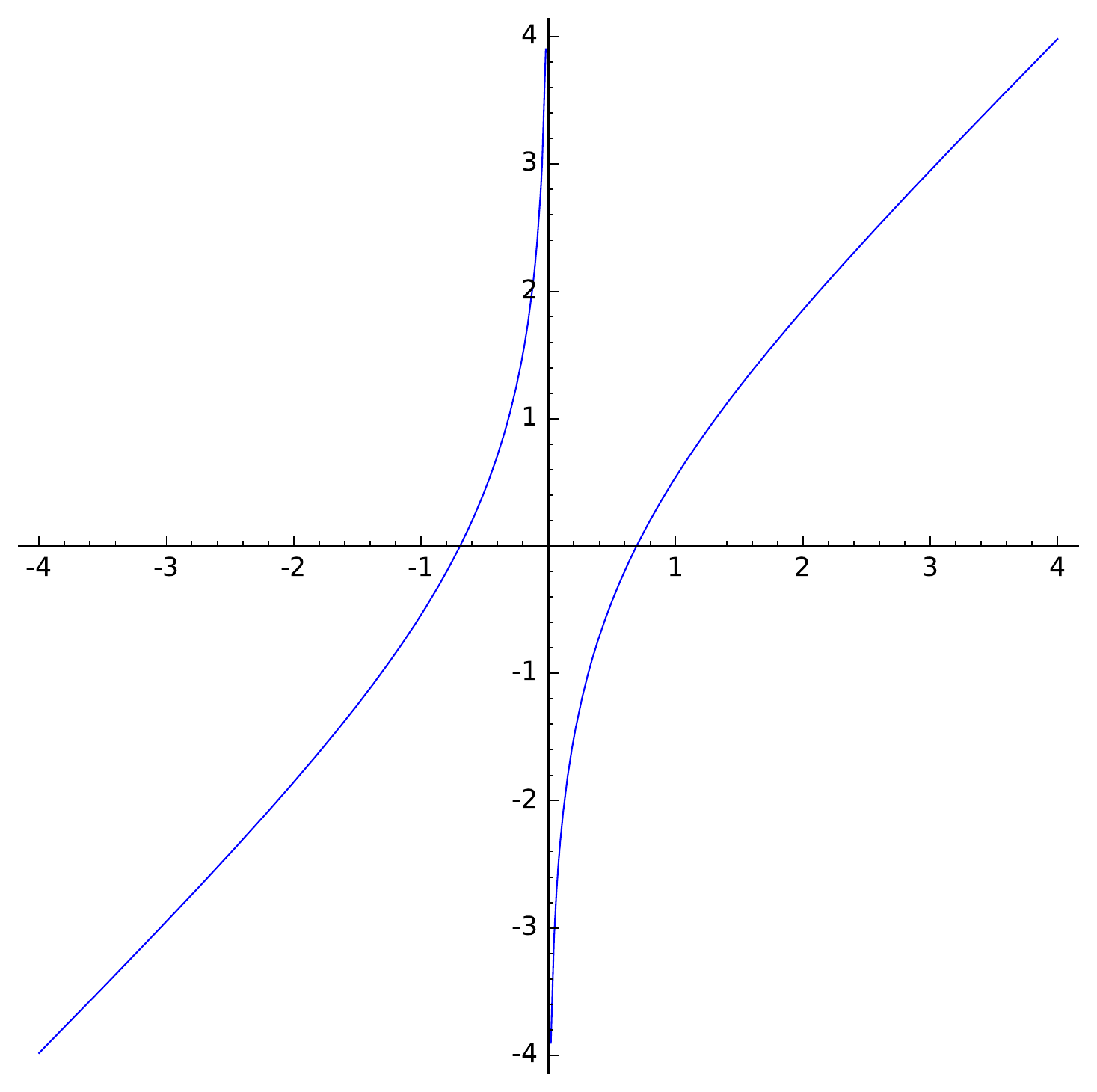}}
\caption{The graphs  of the Lebesgue measure preserving  Farey maps.}
\label{fig.Fareyconj}
\end{center}
\end{figure}

\subsection{Ralston continued fraction}

To study precisely the discrepancy of codings of rotations, D. Ralston \cite{Ra1}, \cite{Ra2} introduced a variant of the Gauss map; if we denote as above the classical Gauss map by $T$ and by $a(x)=[\frac 1x]$ the first partial quotient of $x$, he defined a map $R: [0,1]\to[0,1]$ by $R(x)=T^2(x)$ if $a(x)$ is even, $R(x)=\frac 1{1+T(x)}$ if $a(x)>1$ is odd, and $R(x)=1-x$ if $a(x)=1$.

Some of his results depend on the existence of an absolutely continuous invariant measure, which can be proved by classical techniques, but this is not trivial, since the map is quite complicated (see Fig. \ref{fig.Ralstonmap}); its set of discontinuity points has a countable set of accumulation points.

However, this map satisfies all the hypotheses needed to get a invariant compact set; a tentative model for the natural extension is readily computed, see the right part of Fig. \ref{fig.Ralstonmap}, and shows that the invariant set is defined by the equations $\frac 1{x-1}\le y\le \frac 1{x+1}$ for $0\le x\le \frac12$, and $0\le y\le \frac 1x$ for $\frac 12\le x\le 1$. This implies that the invariant probability density $\phi$ is given by $\phi(x)=\frac {2C}{1-x^2}$ for $0\le x< \frac 12$, and by $\phi(x)=\frac Cx$ for $\frac 12\le x\le 1$, with $C=\frac 1{\log(6)}$.

This could be explicitly proved, either by showing that the Ruelle equation is satisfied (which would be rather tedious, since listing all the antecedents of a point is not simple), or, in a more informative way, by understanding how the images of the various elements of the Markov partition are situated with respect to each other.
 
 \begin{figure}
\begin{center}
\scalebox{0.3}{
\includegraphics{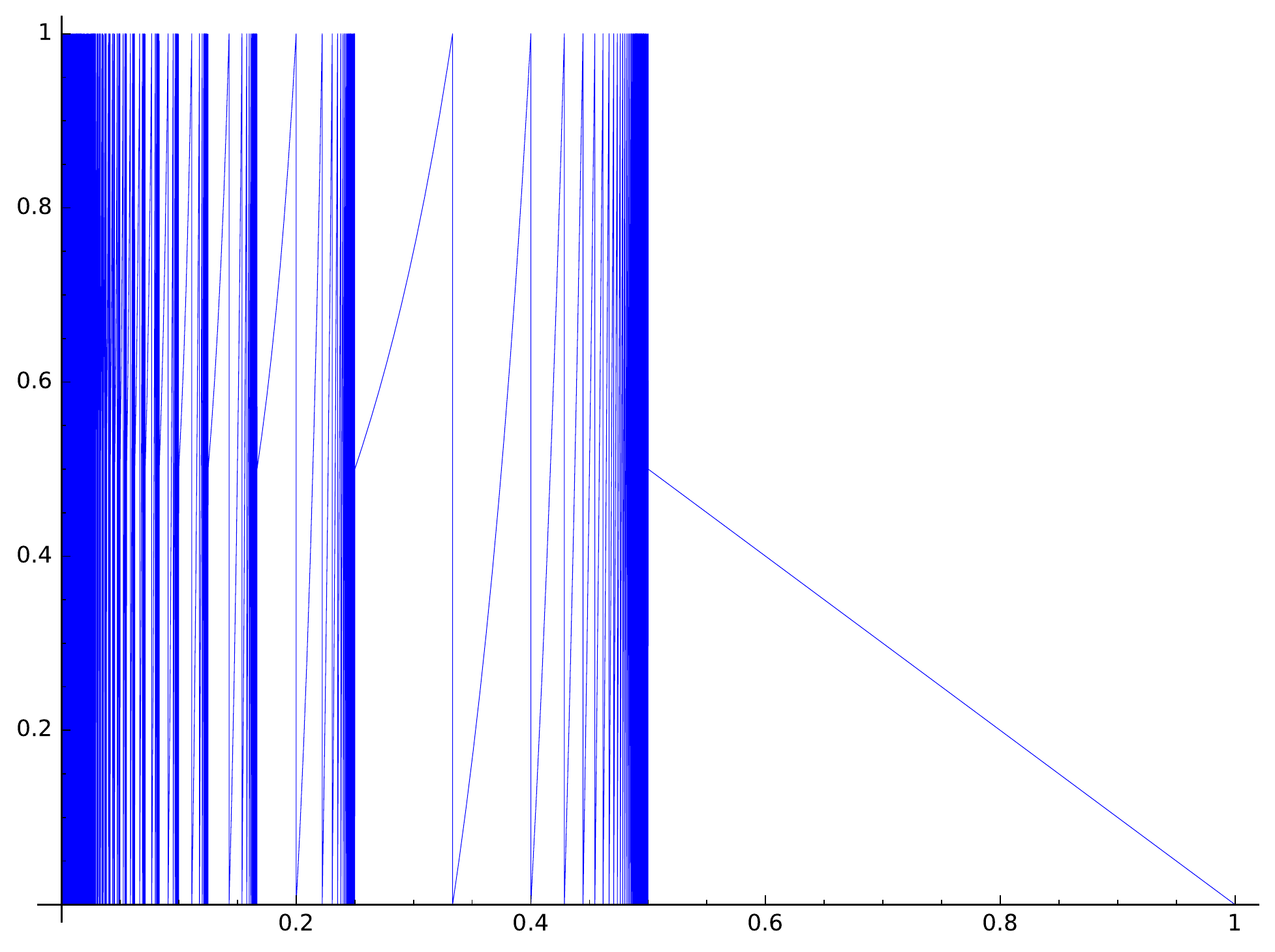} \phantom{a large space} \includegraphics{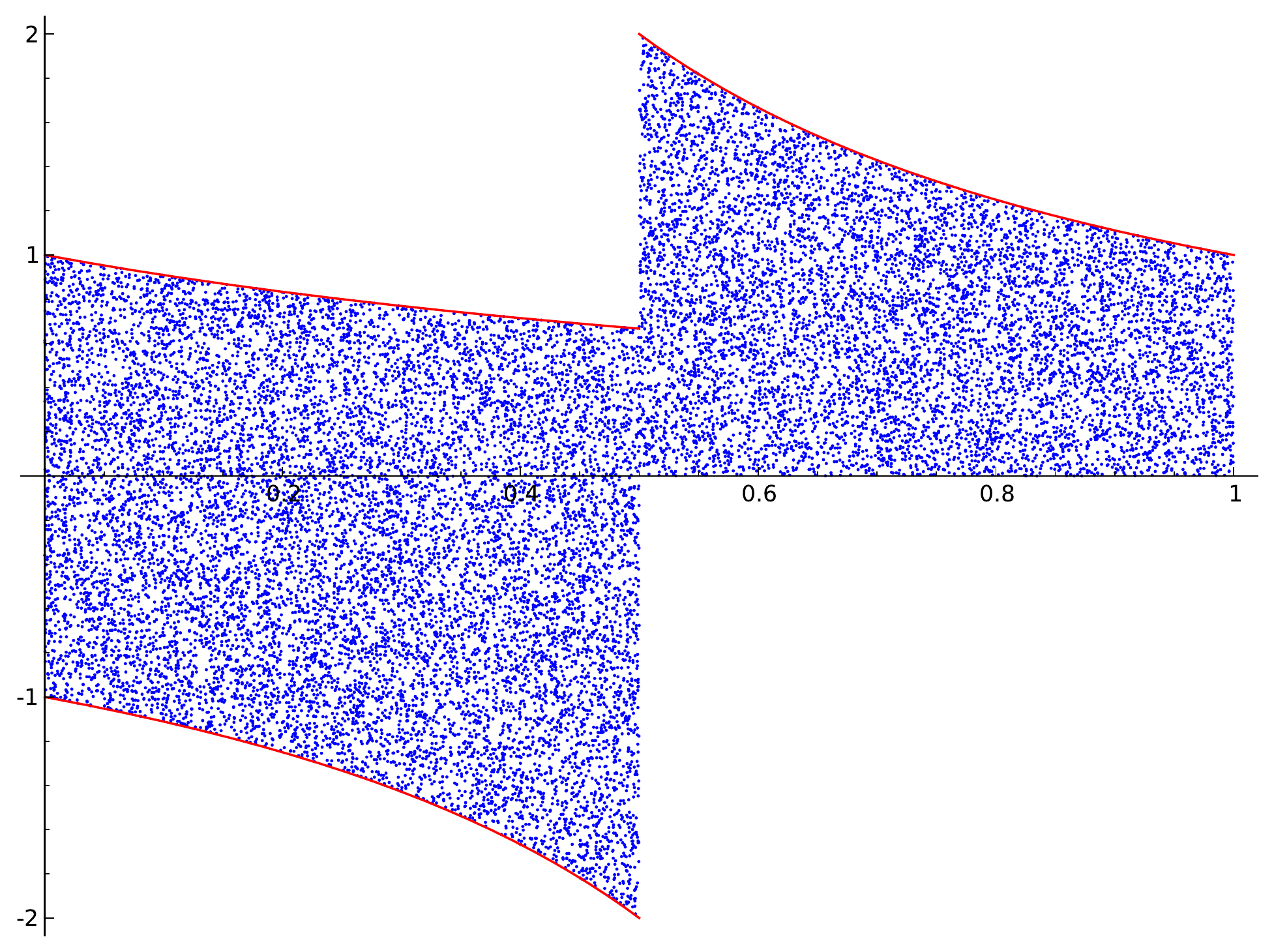}}
\caption{The Ralston map, and $20, 000$ points of an orbit of its natural extension.}
\label{fig.Ralstonmap}
\end{center}
\end{figure}

\subsection{Nakada $\alpha$-continued fractions}   There are many variants of the Gauss map, including those arising from the so-called semi-regular continued fractions.  The best known of these are the nearest integer continued fractions and the backwards continued fractions.   Nakada \cite{N} introduced a very interesting family of continued fractions, depending on a real parameter $\alpha \in [0,1]$ and called the (Nakada) $\alpha$-continued fractions.   On $[\alpha-1, \alpha)$, one defines $T_{\alpha}(x) = 1/|x| - \lfloor 1/|x| + 1 - \alpha\rfloor$.   The values $\alpha =1,1/2, 0$  give  the regular, the nearest integer, and a variant of the backwards  continued fractions, respectively.      

Nakada used planar models of the natural extensions to determine invariant absolutely continuous measures for $T_{\alpha}$, $\alpha \ge 1/2$.   Kraaikamp \cite{K} later confirmed this using his theory of $S$-fractions.     Moussa, Cassa and Marmi \cite{MoussaCassaMarmi:99} extended this to   $\alpha \ge \sqrt{2}-1$.   Luzzi and Marmi  \cite{LuzziMarmi08}  extended this for certain rational values of $\alpha$ in $[0,1]$, in a work that revived interest in these matters.   Carminati and Tiozzo ~\cite{CT} and Kraaikamp-Schmidt-Steiner \cite{KSSt}  gave the result for all $\alpha$.  Central to the study of this family of dynamical systems has been the relation of the entropy of $T_{\alpha}$ to the mass of a planar model of its natural extension.   See the aforementioned papers and references for more on this matter.      

Explicit presentations of (number theoretic versions of) planar models of the natural extensions occur in each of \cite{N}, \cite{LuzziMarmi08}, \cite{CT}, \cite{KSSt}.    Note that the map $T_{\alpha}$ is increasing on the negative portion of its domain and decreasing on the positive;  this mixed nature seems to be the cause of the non-product structure of the natural extension.  Indeed,  it is possible to define an orientation preserving version, replacing $1/|x|$ by $-1/x$, which seems better behaved in this respect.  For example, see the natural extensions of Katok-Ugarcovici's \cite{KatokUgarcovici10} $a,b$-continued fractions.  (See also \cite{CaltaKS}.) 

 \begin{figure}
\begin{center}
\scalebox{0.4}{
\includegraphics{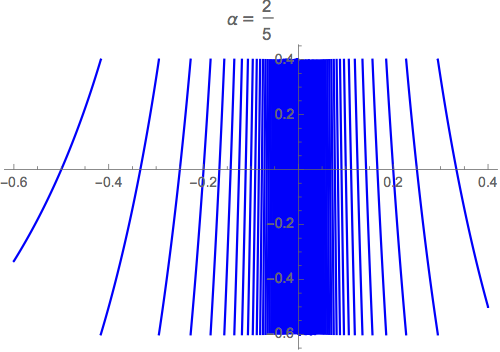} \phantom{a large space} \includegraphics{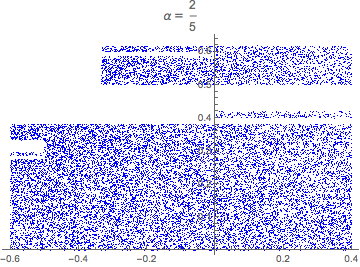}
}
\caption{The Nakada $\alpha$-CF map, when $\alpha = 2/5$, and $20, 000$ points of an orbit of its natural extension.  Note that the domain is not everywhere a (local) product.}
\label{fig.AlpPt39Plot}
\end{center}
\end{figure}

\subsection {Rosen and Veech continued fractions}
There are various ways to associate interval maps to a Fuchsian group.   Rosen~\cite{R} defined a continued fraction for each  the Hecke triangle Fuchsian group.    For $q\in \N, q \geq 3$ and $\lambda = \lambda_{q}= 2 \cos \frac{\pi}{q}$, on the interval $[\, -\lambda/2,\lambda/2\,)\,$ the associated map is 
$T_{q}(x) = \left| \frac{1}{x} \right| - \lambda \left \lfloor\,  \left| \frac{1}{\lambda x} \right| + 1/2\right\rfloor$.  (Thus, when $q=3$ one has the nearest integer map.)   The metric properties of these continued fractions by way of planar natural extensions has seen a great deal of activity in the last 20 years, related work includes \cite{S}, \cite{Nakada92}, \cite{BKS}, \cite{GroechenigHaas96},  \cite{MayerMuehlenbruch10}, \cite{MayerStroemberg08}, \cite{AS}.   
Several authors have also studied one parameter deformations of these and related continued fraction maps,  (called $\alpha$-Rosen maps), see in particular  \cite{DajaniKraaikampSteiner09}. 

In these settings,  the method of this paper can be used to at least gain insight into the possibilities for planar models of the natural extensions of the maps being studied. 

Arnoux and Hubert \cite{AH} introduce continued fractions motivated by the study of Veech groups in the setting of translation surfaces.   Their continued fractions are directly related to subgroups of the Hecke groups that underly the Rosen continued fractions.  In \cite{AS2}, we used natural extensions (informed by the technique of this paper) to compare these two families of continued fractions.    Note that Smillie-Ulcigrai~\cite{SU}  introduced another continued fraction for the study of the geodesics of the translation surface arising from identifying opposite sides of the regular octagon.    

\subsection{Hei-Chi Chan map}

One can show that any irrational number in $[0,1]$ can be written in a unique way:

$$x=\cfrac {2^{-a_1}}{1+{\cfrac{2^{-a_2}}{1+\cdots}}}$$

Using this particular continued fraction, Hei-Chi Chan, in \cite{Ch}, defines a generalized Gauss map with an invariant density of the form $\frac c{(1+t)(2+t)}$. This map admits an additive variant, for which the computations are particularly simple; we define $S_a$ on $[0,1]$ by $S_a(x)=2x$ if $0\le x\le \frac 12$, and $S_a(x)=\frac 1x-1$ if $\frac 12\le x\le 1$. The heuristic leads us to conjecture a natural extension of the form $\tilde S_a(x,y)=(2x,y/2)$ if $0\le x\le \frac 12$, and $\tilde S_a(x,y)=(1/x-1, x-x^2y)$ if $\frac 12< x\le 1$, and an experiment (see Fig. \ref{fig.Heichichanaddmap}) shows that the compact set $\{(x,y)|0\le x\le 1, 0\le y\le \frac 1{1+x}\}$ is invariant by $S_a$.

By a nice coincidence, this additive map has the same invariant domain and the same invariant density as the Gauss map. Note that this map has a fixed point, but it is not indifferent, which explains why this additive map has a finite invariant measure.

 \begin{figure}
\begin{center}
\scalebox{0.3}{
\includegraphics{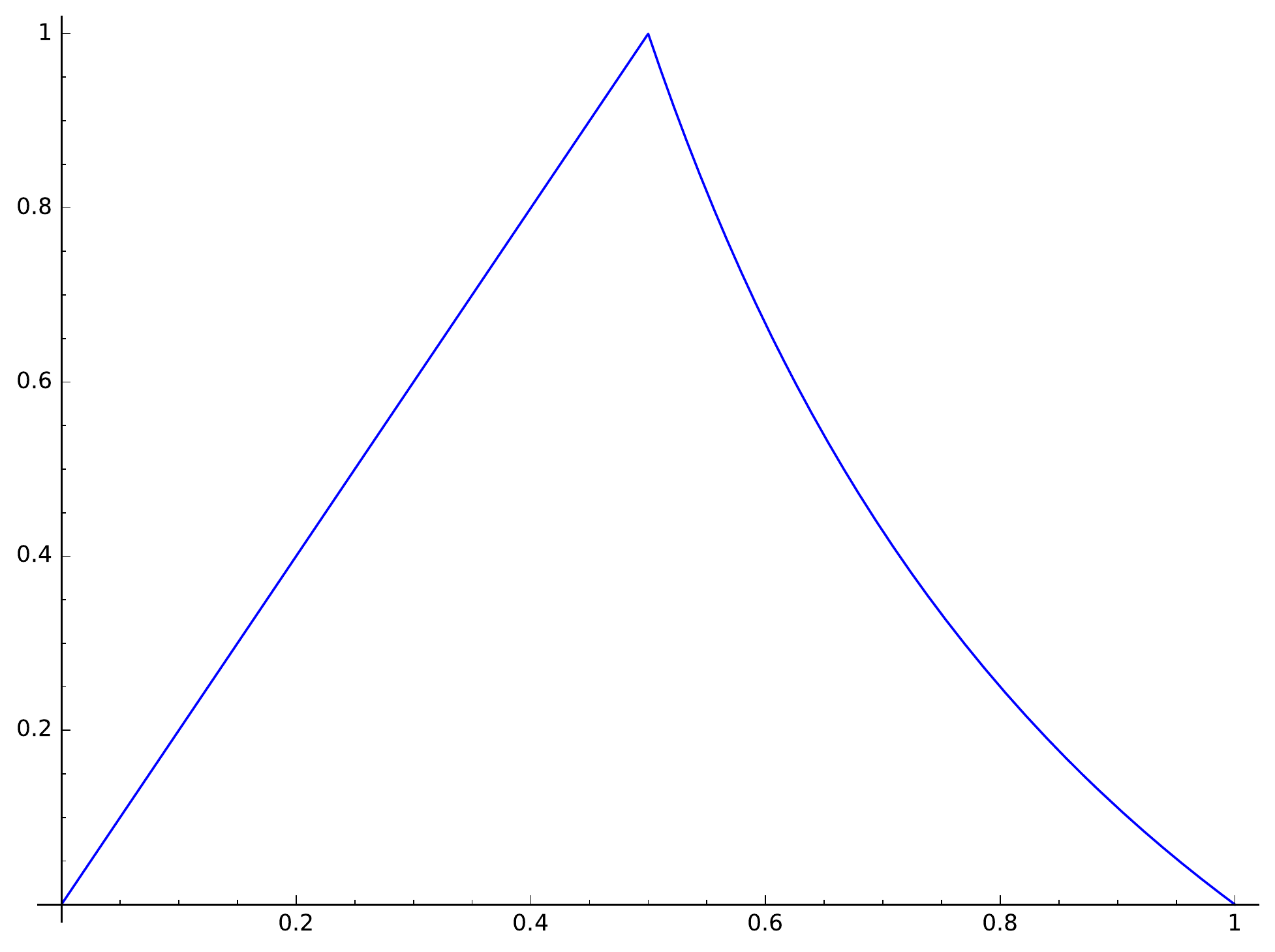} \phantom{a large space} \includegraphics{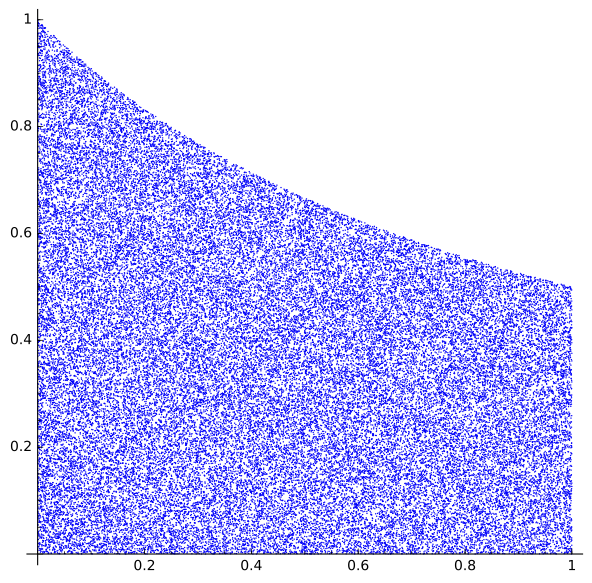}}
\caption{The additive Hei-Chi Chan map, and $20, 000$ points of an orbit of its natural extension.}
\label{fig.Heichichanaddmap}
\end{center}
\end{figure}

The map actually studied by Hei-Chi Chan is  an acceleration of $S_a$; define, for $x\in (0,1)$, $n(x)=\sup\{k\in \N|2^kx\le 1\}$; the multiplicative Hei-Chi Chan map is given by $S_m(x)=\frac 1{2^{n(x)}}-1$. This map has infinitely many decreasing branches, which accumulate exponentially fast to the vertical axis, see Fig. \ref{fig.Heichichanaddmap}.

Numerical experiment shows that the invariant set is given by the equation $\frac1{2+x}\le y\le \frac 1{1+x}$, giving the invariant density $\frac c{(1+x)(2+x)}$ found by Hei-Chi Chan in his paper. The domain of the map $\tilde S_a$ is in fact the image of the Markov box $\frac 12\le x\le 1$, which is in fact natural: it consists in stopping when we have done a series of multiplication by 2 and an inversion, that is, each time we pass the rightmost Markov box.

 \begin{figure}
\begin{center}
\scalebox{0.3}{
\includegraphics{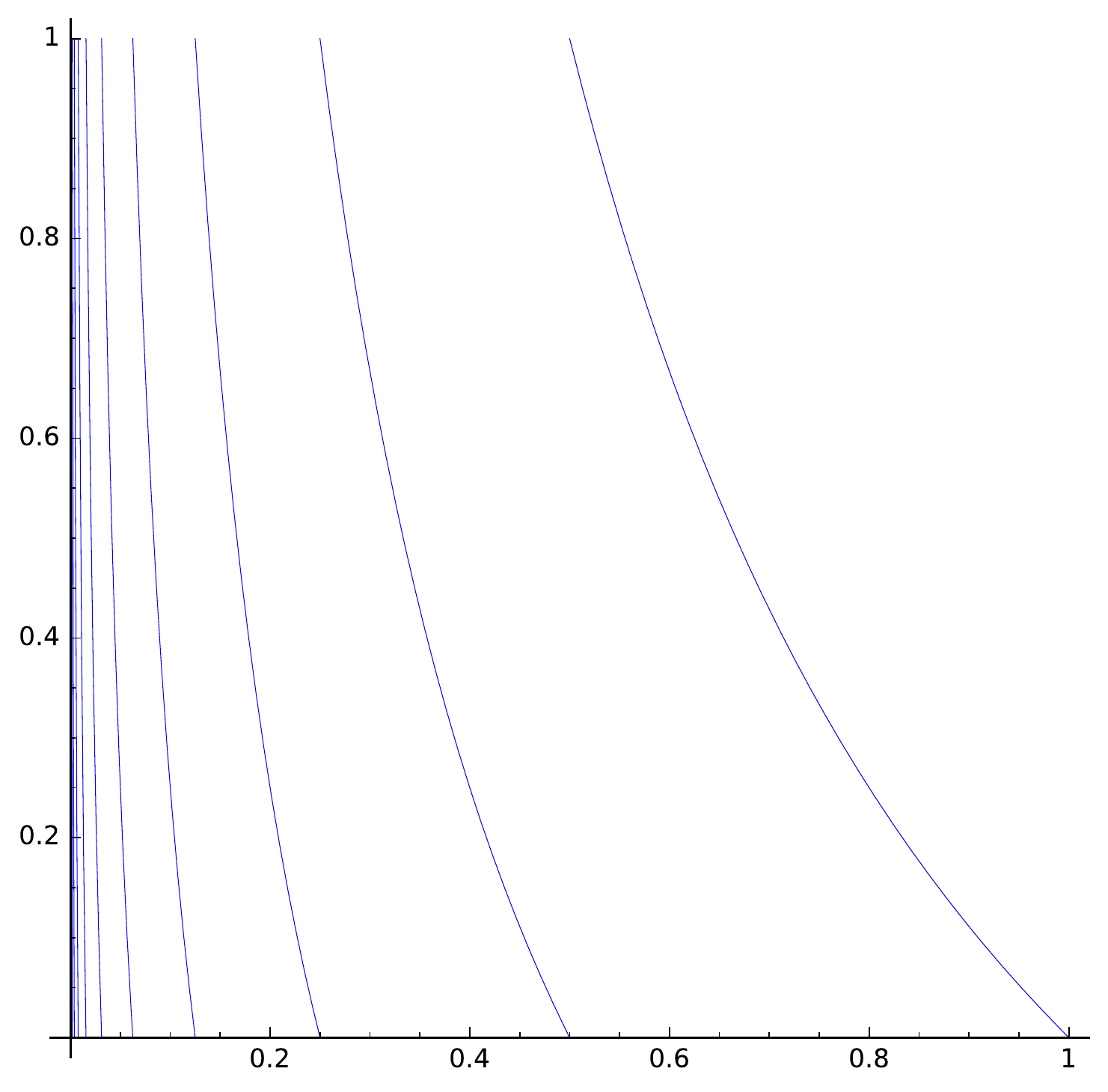} \phantom{a large space} \includegraphics{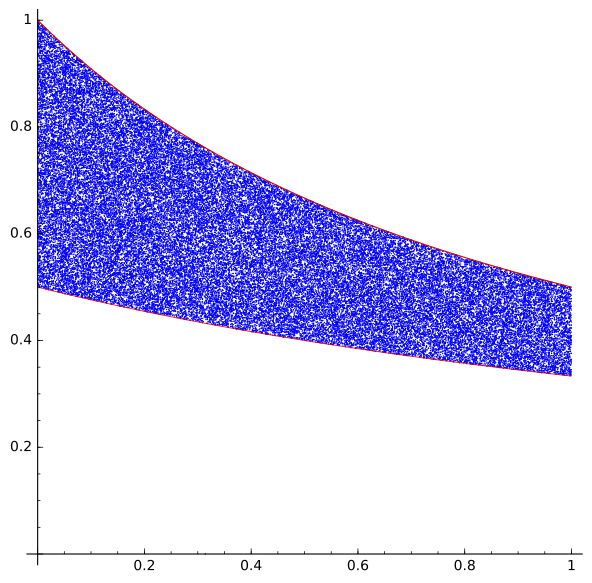}}
\caption{The multiplicative Hei-Chi Chan map, and $20, 000$ points of an orbit of its natural extension.}
\label{fig.Heichichanmultmap}
\end{center}
\end{figure}

\begin{rem} These maps are linked to a curious binary GCD algorithms on integers, defined as follows. Consider two integers $0<p<q$. First divide them by the largest power of 2 that divides both of them; hence we can suppose that one of them is odd, and their GCD is odd. If $2p<q$, and $q$ is even, divide $q$ by $2$. If $2p<q$, and $q$ is odd, multiply $p$ by $2$. If $2p\ge q$, substract $p$ from $q$ and  reorder; continue until $p=0$.

This algorithm preserves the GCD; the value of $q$ is nonincreasing, and it is ultimately decreasing if $p>0$, hence the algorithm terminates with $p=0$ in finite time, the value of $q$ giving the GCD.

This algorithm is a variant of the very efficient binary GCD algorithm studied by Vall\'ee \cite{Va}; the maps $S_a$ and $S_m$ might prove interesting to study the properties of this algorithm.
\end{rem}

\begin{rem}
In opposition to the previous examples, this map does not come from elements of $SL(2,\Z)$ or of a discrete group of $SL(2,\R)$; this does not prevent the heuristics to work efficiently. However, the importance of powers of 2 in the map, as well as the related binary GCD algorithm on the integers, hints to the idea that there should be a 2-adic version of this map and its natural extension.
\end{rem}

\subsection{Hurwitz complex continued fraction}

The Hurwitz continued fraction is simply the extension to the complex plane of the nearest integer continued fraction. More precisely, we consider the Gauss integers $n+mi$, with $n,m$ integers, and for any complex number $z$, following the notations of Nakada, we denote by $[z]_2$ the nearest Gaussian integer (we ignore the measure 0 case when there is more than one such integer). The Hurwitz map $H$ is defined on the unit square $\{z\in \C|-\frac 12<\Im(z), \Re(z)<\frac 12\}$ by $H(z)=\frac 1z-\left[\frac 1z\right]_2$ (similar maps can be defined in an analogous way for other tilings of the complex plane).

 \begin{figure}
\begin{center}
\scalebox{0.3}{
\includegraphics{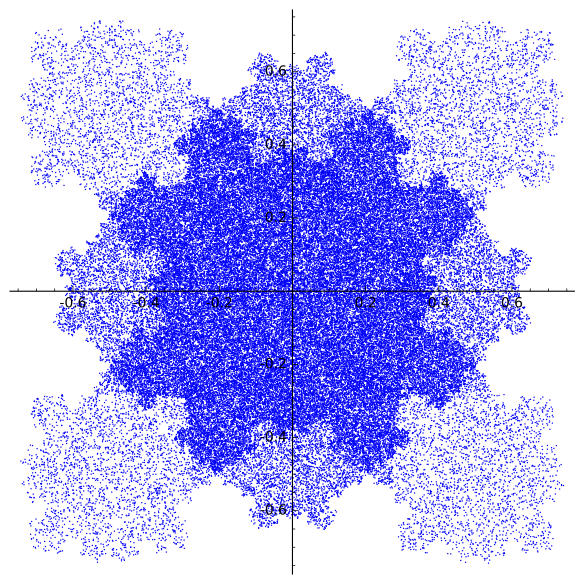} \phantom{a large space} \includegraphics{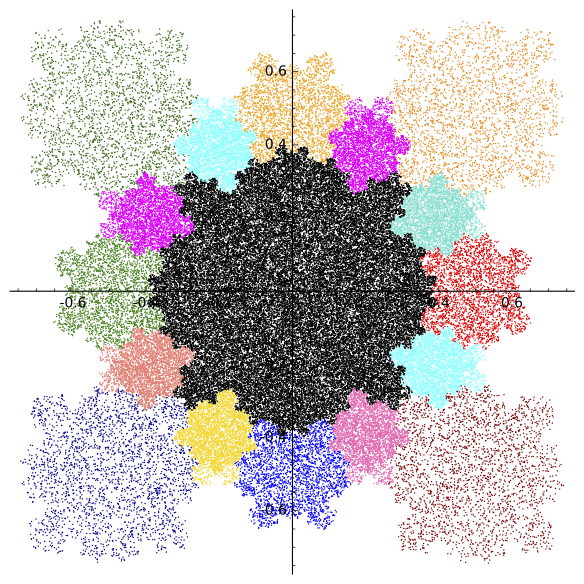}}
\caption{100 000 points of an orbit of the natural extension of Hurwitz map (projected on the second coordinate), colored in the right figure according to the previous partial quotient.}
\label{fig.Hurwitzmap}
\end{center}
\end{figure}

This map and its natural extension has been recently the object of study by Ei, Ito, Nakada and Natsui, see \cite{EINR}; the structure seems quite complicated. One can use the arithmetical version of the natural extension $\tilde H(z,w)=(\frac 1z-\left[\frac 1z\right]_2, \frac 1{w+\left[\frac 1z\right]_2}\,)$, or the conjugate form which preserves Lebesgue measure, and our theorem applies, hence there is a compact invariant set $K$.

It is easy to do numerical experiments, but difficult to interpret their result because they are in dimension 4, and do not allow simple visualization. It seems however that the invariant compact set has nonempty interior, hence gives a natural extension; but the characterization of the domain of the natural extension is not simple, as it seems to have a fractal boundary. Fig. \ref{fig.Hurwitzmap} shows 100 000 points of an orbit, projected on the second coordinate; it seems quite apparent that the density is not the same everywhere. This receives a first explanation when we color the orbit according to the previous partial quotient: the incomplete partial quotients of modulus at most $\sqrt 5$ correspond to outer, and less frequent, pieces of the natural extension.

It seems probable that the family of sets $K_z=\{w\in \C|(z,w)\in K\}$ is locally constant on a finite (17) number of subsets of the unit square. It would be interesting, first to prove it, and then to understand if it still holds when the lattice varies. It would of course be more interesting to determine explicitly the sets $K_z$, their measure and their boundary.

\subsection{Some examples in higher dimension : Brun and Jacobi-Perron}

Our technique also works in higher dimension. Some examples can be found in \cite{AN} and \cite{AL}, where it was applied to recover the explicit formula for the invariant density of several multidimensional algorithms as Brun and Selmer,  and to find this formula for other algorithms such as the reverse algorithm and the Cassaigne algorithm. 
 \begin{figure}
\begin{center}
\scalebox{0.1}{
\includegraphics{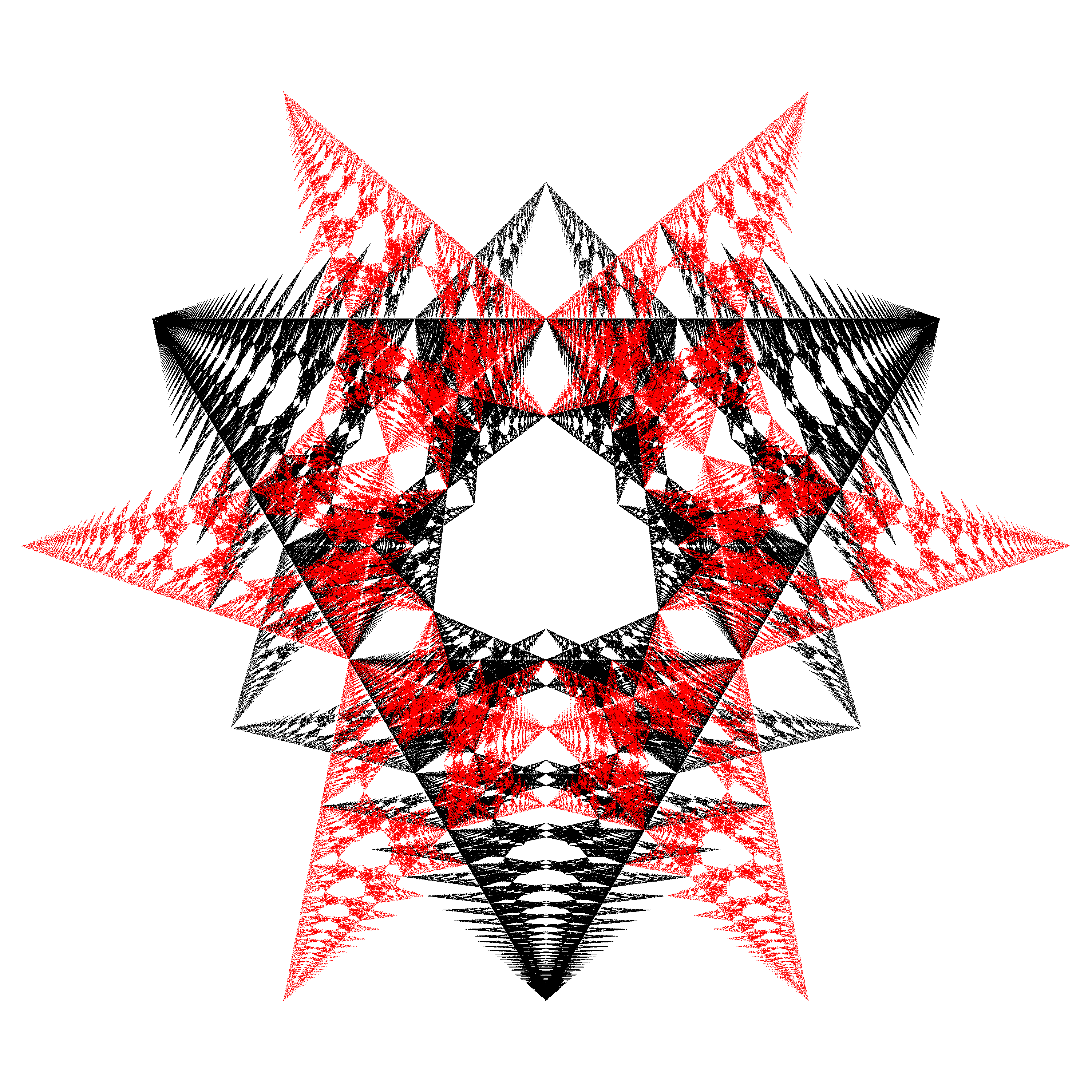}}
\caption{2 000 000 points of an orbit of the conjectured natural extension of the modified Poincar\'e algorithm (projected on the dual coordinates), colored  according to the coding of the preimage. (Courtesy of S\'ebastien Labb\'e.)}
\label{fig.PoincareOrbit}
\end{center}
\end{figure}

Note however that in some cases it fails spectacularly, for example for the classical Jacobi-Perron algorithm, or for the AR-Poincar\'e algorithm (see \cite{BL}); Fig. \ref{fig.PoincareOrbit}, which is due to S\'ebastien Labb\'e shows 2 million points of an orbit. The picture, despite its undeniable esthetic appeal, gives little clue as to  a possible model for the natural extension, and its invariant measure. In particular, it can be seen that sets of points coming from the inner triangle (in red) and from the outer triangles (in black) seem to intersect on a large set, which might explain the degeneracy in this case.

%-------------------------------------------------------------------------------
%section 8 further questions
%-------------------------------------------------------------------------------
\section{Further questions} 

This paper leaves many questions unanswered. The first, and most important,   would be to give sufficient conditions so that the compact invariant set K given by the main theorem has positive measure. In particular, is  there a simple verifiable condition guaranteeing that $K$ and $\widetilde T(K)$, in the notation of Theorem~\ref{th:natext}, share the same positive measure? (One should expect a condition reminiscent of the open set condition for IFS). A first step should be to consider the case when the matrices describing a piecewise homographic algorithm generate a lattice in  $\text{SL}(2, \mathbb R)$; this case is similar to the Bowen-Series interval maps, and one might expect to find the natural extension as a return map of the geodesic flow which factors over the given interval map.

More precisely, if the piecewise homographic map is expanding, it is clear that $cx+d$ is bounded, because $(cx+d)^2<k$ by hypotheses. But $c$ has no reason to be bounded; can we prove that $c(cx+d)$ is always bounded? This is the case in all examples. 
It would also be interesting to find such a condition for positive measure in higher dimension, and when this is the case, to find a classical  object for which there is a flow corresponding to the suspension of the natural extension of the given map.

As we have seen, one also has convergence, albeit to a non-compact set, in some cases of weakly contracting maps, and the theorem should obviously be extendable to those cases, which are of interest since most additive types of continued fractions admit indifferent fixed points. Note also that these indifferent fixed points have different effects in dimension 1, where they lead to unbounded invariant density and to infinite invariant measure, and in higher dimension, where the  invariant density can have a pole while the total measure remains finite. 

Note that our basic example is related to matrices in $SL(2,\Z)$, hence orientation preserving and unimodular. However, we have seen that the heuristic extends quite well to other situations, either orientation reversing ($\alpha$-type continued fraction) or non-unimodular (Hei-Chi Chan continued fraction). Is there a more general setting which includes all these cases? In particular, can we extend this to a $p$-adic setting? Is it true that in the orientation preserving case, the domain of the natural extension is connected?

Finally, to use more efficiently this heuristic method, are there effective ways to visualize and understand the domains we obtain in higher dimension, maybe by animated views of slices of the domain?

%------------------------------------------------------------------------------- 
\appendix
\section{A letter from Gauss to Legendre on continued fractions, translated from the French}

January 30, 1812

Sir, 

I tell you one thousand thanks for the two memoirs you made me the honor to send me and that I received these past few days. The functions you discuss there, as well as the questions of probabilities on which you prepare a large work, have a great appeal for me, although I myself have worked  little  on those last. I remember nevertheless a curious problem, which I considered 12 years ago, but that I could not then solve to my satisfaction. You might interest yourself with it for a little time, in which case I am sure that you will find a more complete solution. Here is the problem. Let M be an unknown quantity between the limits 0 and 1, for which all values are, either equally probable, or more or less following a given law. We suppose it converted in a continued fraction 

$$M=\cfrac {1}{a'+{\cfrac{1}{a''+\text{etc}}}}$$

What is the probability that, stopping the expansion at a finite term $a^{(n)}$, the following fraction 

$$\cfrac {1}{a^{(n+1)}+{\cfrac{1}{a^{(n+2)}+\text{etc}}}}$$

\noindent
be between the limits 0 and $x$? I denote it by $P(n,x)$, and I have, supposing for $M$ all values equally probable, 

$$P(0,x)=x;$$

\noindent
$P(1,x)$ is a transcendental function depending on the function

$$1+\frac 12+\frac 13+\cdots+\frac 1x.$$

\noindent
which Euler calls unexplainable and on which I just gave several researches in a memoir presented to our society of sciences which will soon be printed. But for the cases where $n$ is larger, the exact value of $P(n,x)$ seems intractable. However, I found by very simple reasoning that, for infinite $n$, one has

$$P(n,x)=\frac{\log(1+x)}{\log 2}$$

But the efforts I made during my researches to assign

$$P(n,x)-\frac{\log(1+x)}{\log 2}$$

\noindent
for a very large, but not infinite, value of $n$ have been fruitless.

$$(\cdots)$$

%-------------------------------------------------------------------------------

\end{document}